\documentclass[12pt]{article}
\usepackage[colorlinks=true]{hyperref}

\usepackage[a4paper, total={6in, 8in}]{geometry}

\usepackage{latexsym,amsmath,amsfonts,bm}
\usepackage{constants}
\usepackage{url}

\newtheorem{theorem}{Theorem}
\newtheorem{proposition}{Proposition}
\newtheorem{corollary}{Corollary}

\newtheorem{lemma}{Lemma}

\newtheorem{definition}{Definition}

\newtheorem{example}{Example}

\newtheorem{remark}{Remark}

\newcommand{\eqnref}[1]{(\ref{eq:#1})} 

\newcommand\ignore[1]{}

\newcommand{\funcdef}[5]{\begin{array}{rccccc} #1 & : & #2 & \to     & #3 \\
                                                &   & #4 & \mapsto & #5\end{array}}

\def\R{\mathbb{R}} 
\def\Z{\mathbb{Z}} 
\def\N{\mathbb{N}} 

\newcommand{\Ex}[1]{\mathbb{E}\left[#1\right]} 

\newcommand{\Ind}[1]{\mathbb{I}_{#1}} 

\renewcommand{\Pr}[1]{\mathbb{P}\left(#1\right)} 

\def\sB{\mathcal{B}}\def\sC{\mathcal{C}}
\def\sF{\mathcal{F}}
\def\sG{\mathcal{G}}
\def\sL{\mathcal{L}}
\def\sM{\mathcal{M}}\def\sN{\mathcal{N}}
\def\sP{\mathcal{P}}
\def\sT{\mathcal{T}}
\def\sX{\mathcal{X}}

\def\G{\mathcal{G}}

\newcommand\QED{\ifhmode\allowbreak\else\nobreak\fi
\quad\nobreak$\Box$\medbreak}
\newcommand{\proofstart}{\par\noindent\sl Proof:\rm\enspace}
\newcommand{\proofend}{\QED\par}
\newenvironment{proof}{\proofstart}{\proofend}

\def\eps{\epsilon}

\newcommand{\ER}{\mbox{Erd\"{o}s-R\'{e}nyi}}

\newcommand{\modulo}[1]{\left| #1\right|}
\newcommand{\pnorm}[2]{\left\|#2\right\|_{#1}}
\newcommand{\supnorm}[1]{\pnorm{\infty}{#1}}
\newcommand{\blnorm}[1]{\pnorm{BL}{#1}}

\newcommand{\maxx}[2]{#1 \wedge #2 }

\newcommand{\expp}[1]{\exp \left( #1 \right)} 

\newcommand{\dist}[3]{{\rm d}_{#1}\left(#2,#3\right)} 

\newcommand{\tH}[1]{\overline{\theta}_{#1}}
\newcommand{\oP}{\overline{P}}

\renewcommand{\vec}[1]{\mathbf{#1}}
\newcommand{\vecg}[1]{\boldsymbol{#1}}
\newcommand{\ovl}[1]{\overline{#1}}
\newcommand{\til}[1]{\widetilde{#1}}

\newcommand{\tilEx}[1]{\til{\mathbb{E}}\left(#1\right)}

\usepackage{float}

\usepackage{xr}

\usepackage[english]{babel}

\usepackage[dvipsnames]{xcolor}
\usepackage{afterpage} 

\usepackage{newfloat}
\usepackage{scalerel}[2014/03/10]
\usepackage{stackengine,wasysym}

\DeclareFloatingEnvironment[name={Supp. Figure}]{suppfigure}

\graphicspath{{../Figures/}}

{}  
\usepackage{tikz}

\newcommand{\floor}[1]{\left\lfloor #1 \right\rfloor}

\begin{document}

\title{Interacting diffusions on sparse graphs: hydrodynamics from local weak limits}


\author{Roberto I. Oliveira\thanks{IMPA, Rio de Janeiro, Brazil. 22460-320. \texttt{rob.oliv@gmail.com},~\texttt{rimfo@impa.br}.  supported by a Bolsa de Produtividade em Pesquisa from CNPq, Brazil. His work in this article is part of the activities of FAPESP Center for Neuromathematics (grant \# 2013/07699-0, FAPESP - S. Paulo Research Foundation). }        \and
        Guilherme H. Reis\thanks{IMPA, Rio de Janeiro, Brazil. 22460-320. \texttt{ghreis@impa.br}. Supported by a Ph.D. scholarship from CNPq, Brazil (grant \# 140768/2015-7.)}  \and
        Lucas M. Stolerman \thanks{Department of Mechanical and Aerospace Engineering, University of California, San Diego, CA 92093, USA}
}

\maketitle

\begin{abstract}
We prove limit theorems for systems of interacting diffusions on sparse graphs. For example, we deduce a hydrodynamic limit and the propagation of chaos property for the stochastic Kuramoto model with interactions determined by $\ER$ graphs with constant mean degree. The limiting object is related to a potentially infinite system of SDEs defined over a Galton-Watson tree. Our theorems apply more generally, when the sequence of graphs (``decorated" with edge and vertex parameters) converges in the local weak sense. Our main technical result is a locality estimate bounding the influence of far-away diffusions on one another. We also numerically explore the emergence of synchronization phenomena on Galton-Watson random trees, observing rich phase transitions from synchronized to desynchronized activity among nodes at different distances from the root.   
\end{abstract}


\newpage
\section{Introduction}
\label{intro}


The results of this paper were inspired by a concrete problem. Let $n\in\N$ and $0<p(n)\leq 1$. Define the Erd\H{o}s-R\'{e}nyi random graph $G_n = G(n,p(n))$ as the random graph with vertex set $[n]:=\{1,\dots,n\}$ where two vertices are adjacent with probability $p(n)$, independently of all other pairs. Write $i\sim_{(n)}j$ if $i,j\in [n]$ are adjacent and let $d_i^{(n)}$ denote the degree of $i$ in $G_n$. We consider the stochastic Kuramoto model \cite{kuramoto2012chemical} over each realization of the graph $G_n$, which is defined as a system of interacting diffusions indexed by $i\in[n]$, solutions of the following system of It\^{o} Stochastic Differential Equations (SDEs) in time interval $[0,T]:$
\begin{equation}\label{eq:systembasic}{\rm d}\theta_i^{(n)}(t) = \frac{1}{d_i^{(n)}}\sum\limits_{j\in [n]\,:\,j\sim_{(n)} i} \sin(\theta_j^{(n)}(t) - \theta_i^{(n)}(t)){\rm d}t + \omega_i\,{\rm d}t + {\rm d}B_i(t).\end{equation}
Here the $B_i$ are independent Brownian motions, and the initial positions $\theta^{(n)}_i(0)$ and ``natural frequencies" $\omega_i$ are sampled from some product measure independently from the $B_i$ and $G_n$. We adopt the convention that the first term in the RHS of (\ref{eq:systembasic}) is zero in the case that $d_i^{(n)}=0.$

The following question arises. 


\noindent
\textbf{Problem:} \emph{What is the bulk behavior of this system when $n\to +\infty$ for different choices of $p(n)$?}

More precisely, we want to understand the behavior of the empirical measure of particle trajectories over a time interval $[0,T]$:
\[L_n:= \frac{1}{n}\sum_{v\in [n]}\delta_{\theta^{(n)}_v(\cdot)},\]
	that is a random measure over the space of continuous functions from $[0,T]$ to $\R,$ the space $C([0,T];\R).$
Our problem is potentially interesting because the graph $G_n$ can be very different depending on $p(n)$. For instance, when $n\to +\infty$, $G_n$ is typically connected if $p(n)\gg \log n/n$ and typically disconnected if $p(n)\ll \log n/n$.

As it turns out, all that matters for our problem is the behavior of $np(n)$ as $n\to +\infty$, which is the expected degree of a vertex in $G_n$ (up to a small error). In a recent paper \cite{reis2018}, we proved that $np(n)\to +\infty$  implies that $L_n$ has the same a.s. limit and obeys the same large deviations principle as in the case $p(n)\equiv 1$ of a complete interaction graph. In particular, the limit of $L_n$ is the law of a McKean-Vlasov diffusion, a Markovian process with trajectories in $C([0,T];\R)$. 

In this paper we complement the result for $np(n)\to +\infty$ by describing what happens when $np(n)\to c\in\R_+$. We prove that $L_n$ converges to the law of a non-Markovian process, which is described by a system of the form \eqnref{systembasic} on a potentially infinite Galton-Watson (GW) tree. The mechanism behind this fact is a general theorem relating the {\em local weak convergence of networks} to the hydrodynamics of systems of diffusions on these networks.

\begin{remark}To the best of our knowledge, the preprint version of this paper was the first work to relate systems of interacting diffusions to local weak convergence. A few months later, Lacker et al. independently arrived at more systematic (and essentially more general) results of the same kind. These are discussed in Section \ref{sec:discussion}.\end{remark}

Finally, we numerically investigate synchronization phase transitions for the stochastic Kuramoto Model on GW trees. In particular, we compute  synchronization levels among nodes at different distances from the root, by varying the coupling strength between oscillators, as well as their natural frequencies and initial conditions. In contrast with the full interaction case of the complete graph, we generally observe the emergence of desynchronization phenomena at distant nodes in the sparse setting.

In section \ref{sec:mainresults}, we give an informal description of our main results. In section \ref{sec:organization}, we make comments about the proofs, review past results, and give the outline of the remainder of the paper.

\section{Informal definitions and main results} \label{sec:mainresults}
\subsection{Infinite networks and interacting diffusions}
We will need the concept of a {\em network}. Informally, this is an object of the form
\[N=(G,\vecg{\mu},\vecg{\omega},\vecg{\theta(0)}),\] where:
\begin{enumerate}
\item $G=(V,E)$ is a locally finite graph with countable vertex set $V$ and edge set $E$.
\item $\vecg{\mu}  = (\mu_e)_{e\in E}$ is a vector of positive weights $\mu_e>0$ for the edges of $G$.
\item $\vecg{\omega} = (\omega_v)_{v\in V}$ is a vector of ``media variables" $\omega_v\in\R$ associated with the vertices. 
\item $\vecg{\theta(0)} = (\theta_v(0))_{v\in V}$ is a vector of initial conditions $\theta_v(0)\in \R$ for each vertex.
\end{enumerate}
We wil call $\vecg{\mu} $ the edge marks and $\vecg{\omega},\vecg{\theta(0)}$ the vertex marks. We will say that a network $N$ is {\em finite} if the graph $G$ is finite. We will often abuse notation and write ``$v\in N$" instead of ``$v\in V$". We also write $\mu_{vu}=\mu_{uv}:=\mu_e$ for the weights of pairs $e=\{v,u\}\in E$, and set $\mu_{vu}=0$ if $vu$ is not an edge. 

Suppose $N$ is given. Let $\psi:\R^2\to\R$, $\phi:\R^4\to \R$ be Lipschitz functions with only $\phi$ bounded, and define for each $v\in V$, the total weight 
\[\mu_v:=\sum_{u\in V}\mu_{uv} .\]
Assume that we have a collection $(B_v(\cdot))_{v\in V}$ of i.i.d. standard Brownian motions associated with the vertices of $N$. A {\em system of interacting diffusions on the network $N$} (with this choice of functions $\psi,\phi$) is a random vector 
\[\vecg{\theta}^N(\cdot)=(\theta^N_{v}(\cdot))_{v \in V}\in C([0,T];\R)^{|V|}\] which is a strong solution of the following system of  It\^{o} Stochastic Differential Equations (SDEs):  for each $v \in N$, 
\begin{eqnarray} \label{eq:sdeI} 
{\rm d}\theta^N_{v}(t)&=& \frac{1}{\mu_v}\sum_{u \in V}\mu_{uv}\phi(\theta^N_{u}(t),\theta^N_{v}(t);\omega_v,\omega_u){\rm d}t +\psi(\theta^N_{v}(t);\omega_v){\rm d}t + {\rm d}B_v(t)    ,
\end{eqnarray}
in the time interval $[0,T]$ and with initial conditions $(\theta_v(0))_{v\in N}$. Heretofore, we adopt the convention that the first term in the RHS of (\ref{eq:sdeI}) is $0$ whenever $\mu_v=0$.

When $N$ is finite, our conditions on $\psi$ and $\phi$ are more than sufficient to imply existence and uniqueness for this problem. Our first finding is that the same holds for infinite networks with at-most-exponential growth.

\begin{theorem}[Loose statement of Theorem \ref{theo:exttheta}]\label{thm:loose1} Suppose $N$ is infinite, but there exists a vertex $o\in N$ such that balls around $o$ grow at most exponentially. Also assume that the weights $\mu_{vu}\neq 0$ are uniformly bounded away from $0$ and $+\infty$. Then there exists a system of interacting diffusions over $N$ that is the unique strong solution of (\ref{eq:sdeI}).\end{theorem}

\subsection{Local weak limits and hydrodynamics}

To state our next result, we need the concepts of {\em local metric} and {\em local weak convergence} of networks. Both of these concepts are defined precisely in Section \ref{sec:weak}; for now, we only give an informal description. 

The local metric is defined on {\em rooted networks}, that is, on pairs $(N,o)$ where $N$ is a network and $o\in N$. According to this metric, two networks $(N,o)$ and $(N',o')$ are close~if there are large balls around $o$ and $o'$ where both the graphs and the corresponding marks can be matched nearly exactly. Note that, for this to make sense, we need to consider these networks up to ``rooted isomorphisms"; see Section \ref{sec:weak} for details. 

Now, given a {\em sequence} of finite networks
\[\{N_n = (G_n,\vecg{\mu}_n,\vecg{\omega}_n,\vecg{\theta(0)}_n)\}_{n\in\N},\]
let $o_n$ be a node of $N_n$ chosen uniformly at random, for each $n\in\N$. Let $\sL(N_n,o_n)$ denote the law of the random rooted network $(N_n,o_n)$; all the randomness comes from the choice of $o_n$. We say that $\{N_n\}_{n\in\N}$ converges in the local weak sense to a distribution $\nu$ over rooted networks if the probability laws $\sL(N_n,o_n)$ converge weakly to $\nu$. 

\begin{remark}If we forget about the marks $\vecg{\mu}_n$, $\vecg{\omega}_n$ and $\vecg{\theta(0)}_n$, this is nothing but the better known concept of local weak convergence of graphs. In this case, it is known eg. that $n$-cycles converge to the deterministic rooted graph $\delta_{(\Z,0)}$; that the Erd\H{o}s-R\'{e}nyi graph $G(n,c/n)$ a.s. converges to a Poisson GW tree with parameter $c$; and that random $d$-regular graphs on $n$ vertices a.s. converge to the infinite $d$-regular tree. We show in Section \ref{sec:iidmarks} that if the marks are chosen independently, with each vector $\vecg{\mu}_n,\vecg{\omega}_n,\vecg{\theta(0)}_n$ i.i.d., then the corresponding networks converge a.s. in the local weak sense.\end{remark}

Now note that, for a arbitrary network $N$, if we can define a system of interacting diffusions over $N$ this gives rise to a {\em random network}
\[N^\theta = (G,\vecg{\mu},\vecg{\omega},\vecg{\theta^{N}(\cdot)}),\]
where the initial conditions are replaced by the particle trajectories in the time interval $[0,T]$. We are abusing notation by calling by the name network two different classes of objects. 

Coming back to the sequence of finite networks $(N_n)_{n\in\N}$, the next theorem relates the local weak convergence of $N^\theta_n$ to that of $N_n$. 

\begin{theorem}[Loose statement of Theorem \ref{theo:hydlimit} and Corollary \ref{cor:omitted}]\label{thm:loose2} Assume that $\nu$ is a probability measure on rooted networks which is supported over pairs $(N,o)$ satisfying the assumptions of Theorem \ref{thm:loose1}. Then for almost all samples $(N,o)\sim \nu$ we can solve the system of interacting diffusions (cf. (\ref{eq:sdeI})) as in Theorem \ref{thm:loose1} to consider $(N^\theta,o).$

Now consider a sequence of networks $\{N_n\}_{n\in\N}$, each $N_n$ with $n$ vertices, which converges in the local weak sense to $\nu$. Assume also that the largest vertex degree in $N_n$ is $n^{o(1)}$ for large $n$. Then, almost surely, the sequence $\{N^{\theta}_n\}_{n\in\N}$ of networks marked with the diffusions is locally weakly convergent to the law of $(N^\theta,o)$ when $(N,o)\sim \nu$. As a consequence, the empirical measures
\[L_n :=\frac{1}{n}\sum_{i=1}^n\delta_{\theta^{N_n}_i(\cdot)}\]
converge almost surely to the distribution of $\theta_o^N(\cdot)$ when $(N,o)$ is sampled from $\nu$.\end{theorem}

 Our results also imply a propagation-of-chaos property (see Corollary \ref{theo:chaos}). 

\subsection{Synchronization phenomena and sparsity}


We come back to the particular case of stochastic Kuramoto model. Our results from theorems \ref{thm:loose1} and \ref{thm:loose2} motivate us to explore synchronization phenomena on  finite GW trees, since they appear as the limit object from a sequence of random Erd\"{o}s-R\'{e}nyi graphs.  If we denote $\sT$ as the random GW tree with $m$ vertices rooted at vertex $1$, we consider the system of SDEs: for each $i \in [m]$,
\[{\rm d}\theta_i^{\sT}(t) = K \sum_{j=1}^m a_{ij} \sin(\theta_j^{\sT}(t) - \theta_i^{\sT}(t))  {\rm d}t+\omega_i{\rm d}t + \varepsilon {\rm d}B_i(t),\]
where $\theta_j^\sT(t)$ and $\omega_j$ represent the angular phase and natural frequency of the oscillator indexed by  $j \in \{1,2, \hdots, m\}$, respectively.  The parameter $K \in \mathbb{R_+}$ represents the coupling strength between nodes, and $a_{ij} =1$ if nodes $i$ and $j$ are connected in $\sT$ or $a_{ij} =0$ otherwise. In our numerical analysis, we do not divide the summation over neighbours of $i$ by the degree of $i$. We sample both  initial conditions $\theta_j^{\sT}(0)$ and natural frequencies $\omega_i$ from distinct distributions, and by changing the coupling strength between  nodes, we compute synchronization levels  between the root and those nodes at different distances. In our simulations,  we chose two different models for generating the GW trees:

\begin{enumerate}
\item \textbf{Binomial model:} The offspring is a binomial random variable  with distribution $Bin(n,p)$.
\item \textbf{D-regular model:} The root node has $\mathcal{C}$ children, while the other ones have exactly $\mathcal{C} -1 $ children.
\end{enumerate}

In Section \ref{numeric}, we describe our numerical methods and results in details. Interestingly, we observe how desynchronization emerges among distant nodes, depending on the  choice of the model parameters. These findings enlighten our understanding of synchronization in complex networks and pave the way for new phase transition studies on Kuramoto dynamics.

\section{Comments, references and organization}\label{sec:organization}
\subsection{Comment on proofs}
We now briefly comment on our proofs. The key step is to show that our system (cf. (\ref{eq:sdeI})) satisfies a locality property, Lemma \ref{carne:thecorollary} below. Loosely speaking, this property states that information does not propagate too fast over the graph in systems like \eqnref{sdeI}. To prove this Lemma, we rely on a linear Gronwall argument, which leads to a matrix exponential. A nice wrinkle in the proof is that this exponential can be related to a heat kernel for a random walk over a network, which we can analyze via the Carne-Varoupoulos bound. With this Lemma in hand, our main results follow easily from general principles, including the definition of weak convergence. 

One last comment is that it seems clear that our result is an exemplar of a more general principle. One can gather from our arguments that ``local" systems of particles on graphs should have a ``local hydrodynamic limit" whenever the sequence of underlying graphs converges. In this sense, our main technical contributions consist of formulating this principle precisely and proving the required locality estimate in our setting. 

\subsection{Discussion}
\label{sec:discussion}
As stated above, our motivation was to understand what happens to interacting difusions in the simple case of an Erd\"{o}s-R\'{e}nyi graph with a constant average degree. Our recent preprint \cite{reis2018} showed that the entire regime of a diverging average degree has the same behavior, even at the level of large deviations, as the complete graph (mean-field interactions). Of course, proving an LDP in the setting of the present paper is an interesting topic for further study.

We continue with a very brief review of the literature, referring to \cite[Section 1.2]{reis2018} for more details. The study of our class of systems over complete graphs is a classical topic; see eg. \cite{Sznitman_Chaos,Pra1996} for early results. Recent papers have obtained hydrodynamic limits in settings with singular interactions \cite{lucon2014,lucon2016} or Gaussian couplings and delays \cite{Cabana2013,Cabana2015,Cabana20182}. More recently, several authors \cite{medvedev-meanfield,medvedev-sparse,Delattre2016} have explicitly considered the case of relatively sparse random graphs. A recent preprint by Coppini et al. \cite{coppini2018} obtains an LDP under a stronger degree condition than in our paper \cite{reis2018}, but with otherwise weaker or incomparable assumptions. To the best of our knowledge, the present work is the first paper to explore how interacting diffusions behave over random graphs of constant average degree. 

A few months after the first version of our paper appeared in the Arxiv, Lacker, Ramanan and Wu \cite{Kavita1,Kavita2} posted two preprints where they consider related systems of interacting diffusions over sparse graphs. Outside of our considering weighted edges, their results greatly generalize ours, by allowing the drift and diffusion coefficients of a vertex to depend nonlinearly on the empirical measures of neighboring vertices. Additionally, they make weaker requirements on the sequences of graphs. One technical difference is that our proofs give more quantitative estimates on correlation decay, whereas they rely on ``softer" weak convergence tools. Their preprint \cite{Kavita2} obtain closed-form descriptions of the non-Markovian dynamics of a vertex and its neighbors.

From the synchronization viewpoint, our study introduces novel results for the Kuramoto Model on GW trees. Over the past years, many studies have analyzed synchronization phenomena on various network topologies (\cite{arenas2008synchronization,li2006phase}), yet little attention has been given to sparse random trees. More recently,
Chiba et al. \cite{chibamedvedevmizuhara2018} studied transitions to synchronization for a large family of random graphs,  relating their onset of synchronization and the well-known phase transition for the fully connected network. 
With a more computational approach, Sokolov and Ermentrout \cite{sokolov2018sync} related network structure  with global stability of phase-locked solutions. For power-law random networks, Medvedev and Tang \cite{medvedev2018kuramoto} studied the effects of scale-free connectivity and compared the synchronization thresholds with dense graphs. In contrast with all those recent findings, our analysis on GW-trees allows to investigate the emergence of \emph{desynchronizaton} among  nodes that are distant from the root, which illustrates how  full synchronization is not always achievable by increasing the coupling strength beyond a fixed value.




\color{black}
\subsection{Organization of the paper}
The remainder of the work is organized as follows: Section \ref{sub:preliminaries} reviews notation of functions and measures. In Section \ref{sub:preliminaries} we also present some preliminaries about graphs. Section \ref{sec:weak} reviews networks and local weak convergence. The reader familiar with networks and local weak convergence just need to read this section to know what notation we adopted here. Section \ref{sec:idg} states in full details our main results. 

We prove the Locality lemma in Section \ref{sec:locality}, and the other main results are derived from this lemma in subsequent sections. We solve the infinite system of SDEs  in Section \ref{sec:conttheta}, and we address the  hydrodynamic limit  in Section \ref{sec:proof:theo:hydlimit}.

Finally, in Section \ref{numeric} we present numerical simulations to discuss the synchronization phenomena. Auxiliary results are found in the Appendix, starting at section \ref{app:examples}.

\newpage
\section{Preliminaries}\label{sub:preliminaries}

In this section we fix notation and briefly review some important concepts. 
\subsection{Numbers}
$\N$ is the set of nonnegative integers. For a natural number $n\in\N\backslash\{0\}$, we let $[n]:=\{1,\dots,n\}$. We define the maximum and minimum of two numbers $x,y\in\R$ by $x\vee y$ and $x\wedge y$, respectively. We define $\R_+=\{x\in\R:x>0\}.$

\subsection{Functions and spaces of probability measures} \label{sub:polish} Let $(S,{\rm d})$ be a Polish metric space. We define \[C(S;\R):=\{h:S\to \R\,:\, h \,\mbox{ is continuous} \},\] and for a map $h:S\to \R$, we have the norms:
\begin{eqnarray*}\pnorm{\infty}{h} &:=&\sup_{x\in S}|h(x)|;\\
\pnorm{Lip}{h} &:=&\sup_{x,y\in S\,:\, x\neq y}\frac{|h(x)-h(y)|}{{\rm d}(x,y)};\\
\pnorm{BL}{h} &:=& \pnorm{\infty}{h} + \pnorm{Lip}{h}.\end{eqnarray*}

 We let $\sP(S)$ denote the set of probability measures over (the Borel sets of) $S$. If $X \in S$ is a random element we denote $\delta_X \in \sP(S)$ the Dirac measure at $X$ which is a random measure.

Given a measure $\mu\in\sP(S)$ and a Borel function $h:S\to\R$ we write
\[\mu(h)=\int_S h{\rm d}\mu.\]  

If $X\in S$ is a random element, and $\Ex{\,\cdot\,}$ is the expectation in the probability space that $X$ is defined, we write $\sL(X)\in \sP(S)$ for its law: \[\sL(X)(h)=\Ex{h(X)}.\]
The topology of weak convergence in $\sP(S)$ is metrized by the Bounded-Lipschitz metric defined for $\mu,\nu\in\sP(S)$:
\[{\rm d}_{BL}(\mu,\nu):= \sup\left\{\mu(h)-\nu(h) \,:\, h:S\to  \R,\, \pnorm{BL}{h}\leq 1\right\}.\]

If $X$ and $Y$ are random elements in $S$ defined on the same probability space and $\sL(X)= \mu$ and $\sL(Y)= \nu$ then
\begin{eqnarray} \label{eq:boundonBL} 
{\rm d}_{BL}(\mu,\nu)\leq \Ex{\maxx{{\rm d}(X,Y)}{2}}.
\end{eqnarray}

\subsection{Graphs}

In this paper, a graph $G=(V,E)$ has vertex set $V$ and unoriented edge set $E$. The set $V$ is either finite or countably infinite. We write $x\sim y$ to denote that $xy=yx:=\{x,y\}\in E$. Notice that we allow $x\sim x$ (i.e. a loop edge). The degree $d_x$ of $x\in V$ is the number of $y\in V$ with $y\sim x$. When we need to specify the dependency on $G$ we write $V_G,$ $E_G$, $x\sim_G y$,  and $d_x^G$. We always assume $G$ is {\em locally finite}, i.e. $d_x<+\infty$ for all $x\in V$. We write $|G|$ and $e(G)$ for the number of vertices and edges in $G,$ respectively.

Given a subgraph $H\subset G$ we define
\begin{enumerate}
\item $\partial H=\{v \in H: \exists\, u \,\in G\setminus H \mbox{ with }\ v\sim u \},$
\ignore{acho que nao usamos isso \item $int.H=H\setminus \partial H,$}
\item We write ${\rm dist}(v,u)$ for the distance between $v,u\in V$, i.e., the size of the shortest path between $v$ and $u$ in $G$,
\item ${\rm dist}(v,\partial H)=\inf_{u \in \partial H}{\rm dist}(v,u),$ and
\item For a subset of vertices $H_0 \subset H$ we define
\[{\rm dist}(H_0,\partial H)=\inf_{v \in H_0} {\rm dist}(v,\partial H).\]
\end{enumerate} 

We will also consider weighted graphs. For a graph $G=(V,E)$ the vector $\vecg{\mu}=(\mu_e)_{e\in E}$ is a vector of weights for $G$ if $\mu_e>0$ for any $e\in E.$ To each vector of weights $\vecg{\mu}$ we can associate a matrix $(\mu_{vu})_{v,u \in V}$ such that for $v,u \in V$
\begin{itemize}
\item if $e=\{v,u\}\in E$ then $\mu_{vu}=\mu_{uv}=\mu_e,$ and
\item if $\{v,u\}\notin E$ then $\mu_{vu}=\mu_{uv}= 0.$
\end{itemize}
We write $\mu_v=\sum_{u\in V}\mu_{vu}$ for the total weight of $v$. We say that $(G,\vecg{\mu})$ is a weighted graph and we identify the vector $\vecg{\mu}$ with the associated matrix $(\mu_{vu})_{v,u\in V}$.\\

\subsection{Models of random graphs}\label{sub:randomgraphs}

Some examples of our theory are related to random graph models. Given $n\in\N$, $p\in [0,1]$, the $\ER$ random graph $G(n,p)$ is the random graph with vertex set $[n]$ with no loops, where any two distinct $x,y\in [n]$ are adjacent with probability $p$, independently of all other pairs. We consider (as is customary) sequences of random graphs $G(n,p)$ where $p=p(n)$ may depend on $n$.

Given $n\in \N$ and $\vec{d}=(d_1,\dots,d_n)\in \N^n$, we let $G(n,\vec{d})$ denote a random graph with degree sequence $\vec{d}$, i.e. a graph that is chosen uniformly at random from the set of graphs $G$ with $V(G)=[n]$, no loops, and $d_v^G=d_v$ for each $v\in [n]$. This makes sense only for certain sequences $\vec{d}$. One important particular case is that of random $d$-regular graphs, where $\vec{d}=(d,d,\dots,d)$ for some $d\geq 3$ and we only need to assume $dn$ even.

\section{Local metrics and weak convergence of networks}\label{sec:weak}

In this section, we review the basic aspects of the local topology and local weak convergence of networks. We start with the case of graphs, which is better known. We then discuss the case of networks with more details. Our main references are the survey by Aldous and Steele \cite{aldousS2004}, the lecture notes by Bordenave \cite{bordenaveg}, and the paper \cite{CharleB}.

\subsection{Rooted graphs and local weak convergence}\label{sub:localweak}

When we consider sparse random graphs, we will need to consider their {\em local weak limits}. 

A rooted graph $(G,o)$ consists of a (countable, locally finite) graph $G$ with a distinguished vertex $o\in V_G$. Two rooted graphs $(G,o)$, $(H,p)$ are rooted isomorphic ($(G,o)\cong (H,p)$) if there exists a bijection $f:V_G\to V_H$ mapping $o$ to $p$ and preserving edges. The space $\G^*$ of rooted graphs considered up to isomorphisms can be endowed with a metrizable ``local topology"~that makes it a Polish space. Therefore, we may speak of random elements in this space (we will define a more general metric on networks below). 

Given $r\in\N$, $(G,o)_r$ is the rooted graph with root $o$ that contains the vertices $x\in V_G$ within distance $r$ from $o$, and all the edges between these vertices. $[G,o]_r$ is the equivalence class of $(G,o)_r$. We write $G(v)$ for the connected component of $v$ in $G$. We write $(G(v),v)$ for the graph $G(v)$ rooted at $v$ and $[G(v),v]$ for  its equivalence class up to isomorphism. 

\begin{definition}\label{def:lwlgraphs}
For each finite graph $G\in\sG$ we define the empirical neighbourhood distribution:
\[U(G)=\frac{1}{|V_G|}\sum_{v\in V_G}\delta_{[G(v),v]}.\]
We say that a sequence of finite graphs $G_n$ converges locally weakly to the measure $\rho \in \sP(\G^*)$ if
\[U(G_n)\to \rho \mbox{ in the weak topology of } \sP(\sG^*).\]
If the sequence of finite graphs $G_n$ is random then we say $G_n$ converges almost surely to $\rho$ in the local weak sense if the locally weakly convergence of $G_n$ to $\rho$ holds in a set of probability $1$ with respect to the law of the sequence $G_n$.
\end{definition}
\ignore{ NAO UTILIZAMOS ISSO
Another way to state this property is the following. Take a rooted graph $(G,o)$ with distribution $\rho.$ The local weak convergence of $G_n$ to $\rho$ says that when we select $v_n\in V(G_n)$ uniformly at random, we have that, for any fixed $r$ and any fixed  rooted graph $(H,p)$:
\[\Pr{(G_n,v_n)_r\cong (H,p)_r}\stackrel{n\to +\infty}{\longrightarrow} \Pr{(G,o)_r\cong (H,p)_r}.\]
Equivalently, $G_n$ converges locally weakly to $(G,o)$ if and only if one may couple $(G_n,v_n)$ and $(G,o)$ such that
\[\Pr{(G_n,v_n)_r\cong (G,o)_r}\stackrel{n\to +\infty}{\longrightarrow} 1.\]
We offer some examples of this definition.
}

\begin{example}\label{exe:cycle}Cycle graphs $C_n$ with $n$ vertices locally weakly converge to $\delta_{(\Z,0)}$.\end{example}

\begin{example}\label{exe:er}Suppose that, for each $n$, $G_n$ has the law of the $\ER$ random graph $G(n,c/n)$ for $c>0$ constant. Then for almost all realizations of $G_n$, the sequence $G_n$ locally weakly converges to the rooted GW tree with Poisson offspring distribution with mean $c$.\end{example}

\begin{example}\label{exe:degree}Suppose that for each $n\in\N$ we have a vector \[\vec{d}_n=(d_{n,1},\dots,d_{n,n})\in \N^n.\] Assume the sequence $\vec{d}_n$ has $\max_{1\leq i\leq n}d_{i,n}\leq n^{\eps_n}$ with $\eps_n\to 0$ and the measures
\[P_n:=\frac{1}{n}\sum_{i=1}^n\delta_{d_{i,n}}\]
converge weakly to some $P$ with finite first moment. If we sample $G_n$ from $G(n,\vec{d}_n)$, then for almost all realizations of $G_n,$ $G_n$ locally weakly converges to the unimodular rooted GW tree $UGW(P)$, where the root has offspring distribution $P$, and all other nodes have offspring distribution
\[\widehat{P}(k):= \frac{(k+1)\,P(k+1)}{\sum_{i=1}^{\infty}\,iP(i)}\,\,(k\in\N).\] 
In particular, if $\vec{d}_n=(d,d,d,\dots,d)$ for all $n$, $UGW(P)$ is the infinite (deterministic) $d$-regular tree rooted at a node.\end{example}

\subsection{Rooted networks and local weak convergence}
\label{sec:networks}
 Roughly speaking, a network is a graph $G=(V,E)$ with parameters (or marks) associated to the vertices and edges of $G$. The parameters (or marks) lie in some metric space.

More specifically, let $(\Upsilon,{\rm d}_{\Upsilon})$ and $(\Xi,{\rm d}_{\Xi})$ be two Polish metric spaces. A network $N=(V,E,\vecg{\upsilon},\vecg{\xi})$ is a graph $G=(V,E)$ together with the vectors \[\vecg{\upsilon}=(\upsilon_v)_{v\in V}\in \Upsilon^{|V|} \mbox{ and }\vecg{\xi}=(\xi_e)_{e\in E}\in \Xi^{|E|}\] that gives marks to the vertices and edges of $G$, respectively. We write $\sN_{(\Upsilon,\Xi)}$ for the space of all theses networks with the mark spaces fixed.

 We say that $\ovl{N}=(\ovl{V},\ovl{E},\ovl{\vecg{\upsilon}},\ovl{\vecg{\xi}})$ is a sub-network of $N$ if $(\ovl{V},\ovl{E})$ is a induced sub-graph of $(V,E)$, the vector $\ovl{\vecg{\upsilon}}$ is the restriction of $\vecg{\upsilon}$ to $\ovl{V}$  and $\ovl{\vecg{\xi}}$ is the restriction of $\vecg{\xi}$ to $\ovl{E}.$ In this case we also say that the sub-network $\ovl{N}$ is induced by the sub-graph $(\ovl{V},\ovl{E}).$

When we write a graph property for a network $N=(V,E,\vecg{\upsilon},\vecg{\xi})$ it is implicitly assumed that this property holds for the underlying graph. For example, $V_N:=V_G$, $E_N:=E_G$ and for a vertex  $v\in N$, $d_v^N=d_v^G$. The boundary of a sub-network $\ovl{N}$ of $N$ is the boundary of the corresponding graphs.

Consider two networks $N=(V,E,\vecg{\upsilon},\vecg{\xi})$ and $N'=(V',E',\vecg{\upsilon}',\vecg{\xi}')$ belonging to $\sN_{(\Upsilon,\Xi)}$. A network isomorphism $\Psi$ between $N$ and $N'$ is a bijection 
$\Psi:V\to V'$ between the vertex sets that preserves edges and marks:
\begin{itemize}
\item $\{u,v\}\in E$ if and only if $\{\Psi(u),\Psi(v)\}\in E',$
\item $\upsilon'_{\Psi(u)}=\upsilon_u$, $\forall u \in V,$ and $\xi'_{\{\Psi(u),\Psi(v)\})}=\xi_{\{u,v\}},$ $\forall \{u,v\}\in E.$
\end{itemize}

A rooted network $(N,o)$ is a network $N$ with a distinguished vertex $o$. Two rooted networks $(N,o)$ and $(N',o')$ are rooted isomorphic if there is a network isomorphism that sends $o$ to $o'.$ 

Given a rooted network $(N,o)$ and a radius $r\in \N$ let $(N,o)_r$ be the network induced by $(G,o)_r$ rooted at the vertex $o.$ Sometimes we identify the rooted network $(N,o)_r$ with its underlying network (without the root).

For a rooted network $(N,o)$ with  $N\in \sN_{(\Upsilon,\Xi)}$ we associate its equivalence class $[N,o]$ of rooted isomorphism.  Define \[\sN^*_{(\Upsilon,\Xi)}=\{[N,o]: \,(N,o)\,\mbox{ is a rooted network with mark spaces }(\Upsilon,\Xi)\}.\]

We now define a notion of distance over rooted networks up to isomorphism. This is not the exact same notion as in  \cite{bordenaveg}, but it is equivalent to it, as a simple calculation shows. 

Consider two rooted networks $(N,o)=(V,E,\vecg{\upsilon},\vecg{\xi},o)$ and $(N',o')=(V',E',\vecg{\upsilon}',\vecg{\xi}',o')$ belonging to $\sN^*_{(\Upsilon,\Xi)}$. Given $r\in \N$ and $\delta>0$, we say that the pair $(r,\delta)$ is good for $(N,o),(N',o')$ if there exists a rooted isomorphism $\Psi$ between $(N,o)_r$ and $(N',o')_r$ such that the corresponding marks are close by $\delta$: 
\begin{itemize}
\item ${\rm d}_\Upsilon(\upsilon_v,\upsilon'_{\Psi(v)})< \delta$, $\forall\ v \in (N,o)_r,$ and
\item ${\rm d}_\Xi\left(\xi_{\{u,v\}},\xi'_{\{\Psi(u),\Psi(v)\}}\right)< \delta,$ $\forall\ \{u,v\}\in (N,o)_r.$
\end{itemize}
The distance between $[N,o]$ and $[N',o']$ is defined by
\begin{align}\label{def:distN}
\dist{\sN^*_{(\Upsilon,\Xi)}}{[N,o]}{[N',o']} = \inf\left\{\frac{1}{1+r} + \delta\,:\,(r,\delta)\mbox{ is good for $(N,o),(N',o')$}\right\}.
\end{align} 

One can show that $(\sN^*_{(\Upsilon,\Xi)},{\rm d}_{\sN^*_{(\Upsilon,\Xi)}})$ is Polish (cf. \cite{bordenaveg}).

 Given two classes of rooted networks, we can define the distance between them using any representatives of these classes. This is well defined since the distance between rooted networks in invariant up to rooted isomorphism.

Sometimes we identify a rooted network with its equivalence class.

For a network $N$ and a vertex $v \in V$ we associate the rooted network $(N(v),v)$ that is the network induced by the connected component of $v$ rooted at $v$.
 
We write $[N(v),v]$ for the equivalence class of the rooted network $(N(v),v)$. For  a finite network $N$ with vertex set $V$ we define the empirical measure neighbourhood:
\[U(N)=\frac{1}{|V|}\sum_{v\in V}\delta_{[N(v),v]}.\]

\begin{definition}[Local weak convergence]\label{def:localweak} Consider a sequence of networks \[N_n=([n],E_n,\vecg{\upsilon}_n,\vecg{\xi}_n)\in\sN_{(\Upsilon,\Xi)}.\] Let $\rho \in \sP\left(\sN^*_{(\Upsilon,\Xi)}\right)$. We say that $N_n$ converges locally weakly to $\rho$ if
\[U(N_n)\to \rho\]
in the sense of weak convergence. If the sequence of finite networks $N_n$ is random then we say that $N_n$ converges almost surely to $\rho$ in the local weak sense if the locally weakly convergence of $N_n$ to $\rho$ holds in a set of probability $1$ with respect to the law of the sequence $N_n$.\end{definition}

\subsubsection{Networks with i.i.d. marks that converge locally weakly}
\label{sec:iidmarks}
Now we give examples of networks that satisfy Definition  \ref{def:localweak}.

Let $G=(V,E)$ be a graph. In Section \ref{sec:mainresults} we introduced the networks
\[N=(G,\vecg{\mu},\vecg{\omega},\vecg{\theta(0)})\]
where we have the vector of weights $\vecg{\mu}=(\mu_e)_{e\in E}\in \R_+^{|E|}$ that are marks for the edges of $G$, environment or ``media" variables $\vecg{\omega}=(\omega_v)_{v\in V} \in \R^{|V|},$ and a vector of initial conditions $\vecg{\theta(0)}=(\theta_v(0))_{v\in V}\in\R^{|V|}$ that are marks for the vertices of $G$. 

\begin{definition}\label{def:spaceofnetworks} We write $\sN$ for the collection of networks $N=(G,\vecg{\mu},\vecg{\omega},\vecg{\theta(0)})$ with edge marks in $\R_{+}$ and vertex marks in $\R\times\R$. When we distinguish a root for $N$ we write $\sN^*$ for the collection of rooted networks up to isomorphism.  
\end{definition}

Our first goal is to show that for a sequence of graphs $(G_n)_{n\in\N}$ and under some additional conditions we can construct interesting examples of random networks $N_n\in\sN$ and a probability measure $\nu\in\sP(\sN^*)$ such that the sequence $(N_n)_{n\in\N}$ converges to $\nu$ in the local weak sense, for almost all realization of the marks.

For any fixed graph $G=(V,E)$ we construct a probability space where we can define the vectors \[\vecg{\mu}=(\mu_e)_{e\in E},\vecg{\omega}=(\omega_v)_{v\in V},\,\mbox{and }\vecg{\theta(0)}=(\theta_v(0))_{v\in V} \mbox{ and they are independent}.\]

When we need to explicit the dependency on $G$ we write $\vecg{\mu}^G,$ $\vecg{\omega}^G,$ and $\vecg{\theta(0)}^G.$ 

Fix the measures $\pi,\lambda\in\sP(\R)$. For a fixed vertex $v$, $\pi$ is the distribution of the media variables $\omega_v$ and $\lambda$ is the distribution of the initial conditions $\theta_v(0).$ Fix a measure $\mu\in\sP(\R_+)$ for the distribution of the weights $\mu_{e}$ for an edge $e\in E.$

We define by $N^G\in\sN^*$ the random network obtained from $G$ by adding these random marks.

In this way we have a transition kernel that associate to each rooted graph $(G,o)$ the law of the random rooted network $(N^G,o)$
\[\funcdef{\sM}{\sG^*}{\sP(\sN^*)}{(G,o)}{\sL(N^G,o)},\]
that is, $\sM(G,o)(h)=\tilEx{h(G,o,\vecg{\mu}^G,\vecg{\omega}^G,\vecg{\theta(0)}^G)}$ where $\tilEx{\cdot}$ is the expectation of the probability space where the marks for $G$ were defined and $h:\sN^*\to\R$ is a bounded measurable function.

The proofs of the following results are given in the appendix because they are easier version of our main results (Theorem \ref{theo:exttheta} and Theorem \ref{theo:hydlimit}, respectively).
\begin{proposition}[Proof in Section \ref{proof:prop:mcontinuity}]\label{prop:mcontinuity} The transition kernel $\sM$ is continuous.
\end{proposition}

With this result at hand, for any probability measure $\rho$ over $\sG^*$ we can define $\rho\sM\in\sP(\sN^*)$ via the formula
\[\rho\sM(h)=\int\sM(G,o)(h){\rm d}\rho\]
for any bounded measurable function $h:\sN^*\to\R.$

\begin{theorem}[Proof in Section \ref{proof:theo:lwlexamples}]\label{theo:lwlexamples} Consider a sequence $\{G_n\}_{n\in\N}\subset \sG$ of finite graphs where each $G_n$ has vertex set $[n]$. Assume also the following
\begin{enumerate}
\item {\em Local weak convergence:} $\{G_n\}_n$ is locally weakly convergent to a probability measure $\rho$ (in the sense of Definition \ref{def:lwlgraphs}).  
\item {\em Small maximum degree:} for each $n$, $\max_{v\in [n]}d_v^{G_n}=n^{\eps_n}$ where $\eps_n\to 0$ as $n\to +\infty$.
\end{enumerate}
Then the sequence of random networks $\{N_n\}_n\subset \sN$ defined by $N_n:=N^{G_n}$, containing the random marks, converges to $\rho\sM$ in the local weak sense and almost surely with respect to the law of the random marks.\end{theorem}

\section{Interacting diffusions on finite networks}
\label{sec:idg}
In this section, we introduce our main objects of interest, and give formal versions of Theorem \ref{thm:loose1} and Theorem \ref{thm:loose2}.
\ignore{
Let $G=(V,E)$ be a graph. Let $\vecg{\mu}$ be a matrix of weights for $G$. We want to introduce a collection of variables indexed by the vertices of $G$ that we will call interacting diffusions on $G$. For this we need the environment or ``media" variables $\vecg{\omega}=(\omega_v)_{v\in V} \in \R^V$. We also need a vector of initial conditions $\vecg{\theta}(0)=(\theta_v(0))_{v\in V}\in\R^V$. With these marks we define the network $N=(G,\vecg{\mu},\vecg{\omega},\vecg{\theta}(0)).$ Observe that the mark space of $N$ is $\R_{+}\times \R^2.$
\begin{definition} We write $\sN$ for the collection of networks $N=(G,\vecg{\mu},\vecg{\omega},\vecg{\theta}(0))$ with mark space $\R_{+}\times \R^2$ with the additional condition that $(G,\vecg{\mu})$ is a weighted graph. When we distinguish a root for $N$ we write $\sN^*$ for the collection of rooted networks up to isomorphism.  
\end{definition} 
}

Fix functions $\psi:\R^2 \to \R$ and $\phi : \R^4 \to \R$. In our model, $\psi$ represents a drift term that depends on the current position and the media variable, whereas $\phi$ corresponds to pairwise interactions that represent single-term drift and the interaction between the particles. As we can always suppose $V\subset \N$ we consider $(B_v)_{v \in \N}$ a collection of i.i.d. Standard Brownian Motions.

\begin{definition} \label{def:supersparseID} Let $N=(G,\vecg{\mu},\vecg{\omega},\vecg{\theta}(0))$ be a network belonging to the set $\sN$ in Definition \ref{def:spaceofnetworks}. A system of interacting diffusions on the network $N$ (with the choice of functions $\psi,\phi$) is a random vector \[\vecg{\theta}(\cdot)=(\theta_{v}(\cdot))_{v \in V}\in C([0,T];\R)^{ |V|}\] which is a strong solution of the following system of  It\^{o} Stochastic Differential Equations (SDEs):  for each $v \in V$, 
\begin{eqnarray} \label{eq:sde} 
{\rm d}\theta_{v}(t)&=& \frac{1}{\mu_v}\sum_{u \in V}\mu_{uv}\phi(\theta_{u}(t),\theta_{v}(t);\omega_v,\omega_u){\rm d}t+\psi(\theta_{v}(t);\omega_v){\rm d}t + {\rm d}B_v(t),
\end{eqnarray}
where the first term in the RHS in (\ref{eq:sde}) is zero if $\mu_v=0,$ and the initial conditions are given by $\vecg{\theta}(0)=(\theta_v(0))_{v\in V}$. When we need to make explicit the dependency on the network we will write $\vecg{\theta}(\cdot)=:\vecg{\theta}^N(\cdot)$.\end{definition}

When the network $N$ is finite the standard theory of It\^{o} SDEs guarantees that the system  (\ref{eq:sde}) has a unique strong solution with continuous trajectories whenever $\psi$ and $\phi$ are Lipschitz-continuous (see \cite{karatzas2012brownian}, Theorem 2.9, Chapter 5). In the remainder of this work, $\psi$ is assumed to be Lipschitz-continuous and $\phi$ is assumed to be Lipschitz-continuous and bounded. We will argue that, under certain conditions, we also can solve (\ref{eq:sde}) simultaneously with infinite equations (on infinite networks).

For each network $N$, if there exists a system of interacting diffusions over $N$, we can replace the initial conditions $\vecg{\theta(0)}$ by the vector of random continuous functions $\vecg{\theta}^{N}(\cdot)\in C([0,T];\R)$ as new marks. This brings us to the following definitions. 

\begin{definition}\label{def:newmarks} Given a network $N=(G,\vecg{\mu},\vecg{\omega},\vecg{\theta(0)})\in\sN$ and a vector of continuous functions \[\vecg{\alpha}(\cdot)=(\alpha_v(\cdot))_{v\in G}\in C([0,T];\R)^{|G|}\] we write $\sC$ for the collection of networks $(G,\vecg{\mu},\vecg{\omega},\vecg{\alpha}(\cdot))$ with continuous functions replacing the vector $\vecg{\theta(0)}$ as new marks. That is, $\sC$ has edge marks in $\R_+$ and vertex marks in $\R\times C([0,T];\R).$ When we distinguish a root for $(G,\vecg{\mu},\vecg{\omega},\vecg{\alpha}(\cdot))$ we write $\sC^{*}$ for the collection of these rooted networks up to isomorphism.\end{definition}

\begin{definition}\label{def:Ntheta} When $N\in \sN$ is such that a system of interacting diffusions $\vecg{\theta}^N(\cdot)$ can be defined, we let $N^{\theta}=(G,\vecg{\mu},\vecg{\omega},\vecg{\theta}^{N}(\cdot))$ be the corresponding random element of $\sC$ defined above. Note that the law of $N^{\theta}$ is invariant by network isomorphisms.\end{definition}

To present our main results, we need some additional definitions. We first restrict ourselves to the space of finite networks.  

\begin{definition} Let $\sN_f$ be the collection of finite networks contained in $\sN$. We write $\sN_f^*$ when considering finite rooted networks of $\sN^*$ up to isomorphism. 
\end{definition}

Finally, we define a map associating to each finite network $[N,o]$ the law of $[N^{\theta},o]$.

\begin{definition}\label{def:Theta}For each rooted finite network $[N,o]\in \sN^*_f$, we let $\Theta([N,o])\in \sP(\sC^*)$ denote the law of the random network $[N^{\theta},o]$ (up to isomorphism). This is well-defined because the law of $N^{\theta}$ is invariant under isomorphisms of $N$. Therefore, $\Theta$ defines a map from $\sN^*_f$ (the space of finite networks) to $\sP(\sC^*)$ (the space of probability measures over $\sC^*$).\end{definition} 

\ignore{Our goals consist in the following. First, we show that $\Theta$ can be extended to a continuous function defined in a closed subset of infinite networks in $\sN^*$.
After that we will consider a sequence $N_n\in\sN_f$ such that the empirical neighborhood distribution $U(N_n)$ converges to $\nu$ in $\sP(\sN^*).$ We will investigate the limit behaviour of the empirical neighborhood distribution  $U(N_n^{(\theta)})$ of the networks obtained by adding the solutions of (\ref{eq:sde}) as new marks. Observe that $U(N_n^{(\theta)})$ is a random element in $\sP(\sC^*).$}

\subsection{Main results}
\label{sec:vsparse:main}
In this section we state our main results. Our results hold when the sequence $\{N_n\}_{n\in\N}$ converges to something ``nice".

\begin{definition}[``Nice" networks]\label{def:goodnet} For each rooted network $(N,o)\in\sN^*$ define $\mu_*(N)$ (resp. $\mu^*(N)$) as the infimum (resp. supremum) of the set  $\{\mu_{e}:e\in E_N\}.$ Let $(G,o)$ be the underlying rooted graph of $(N,o).$ Call $(N,o)$ {\em nice} if:
\begin{enumerate}
\item {\em Graph grows at most exponentially:} there exists $a>0$ such that \[|\partial (G,o)_r|\leq ae^{ar} \mbox{ for all }r\geq 1; \mbox{ and}\]
\item  {\em Weights bounded away from $0$ and $+\infty$:} $0<\mu_*(N)\leq \mu^*(N)<+\infty.$
\end{enumerate}  
We let $\sB^*\subset \sN^*$ denote the set of nice rooted networks (this is a Borel set). Observe that $\sN_f^*\subset\sB^*.$ \end{definition}

See Section \ref{sec:goodexamples} for examples of nice networks.

Our first main result ensures that systems of interacting diffusions may be defined over nice networks $(N,o)$ as limits, taking $r\to\infty$, of systems over the finite networks $(N,o)_r$. 

\begin{theorem}[Extension of $\Theta$, Proof in Section \ref{sec:conttheta}]\label{theo:exttheta} We can extend $\Theta$  to a continuous transition kernel $\ovl{\Theta}:\sB^*\to\sP(\sC^{*})$. More specifically, given $(N,o)\in \sB^*$ there exists the unique strong solution of the (possibly infinite) system of SDEs in (\ref{eq:sde}) defined over $N$. Furthermore,  $\ovl{\Theta}(N,o)$ is the law of $(N^\theta,o)\in\sC^*$ defined replacing the initial conditions $\vecg{\theta(0)}$ by $\vecg{\theta}^N(\cdot)$.
\end{theorem}

In our next result, we use Theorem \ref{theo:exttheta} to identify  hydrodynamic limits of systems of interacting diffusions. We first note that, if $\nu$ is a probability measure over $\sB^*,$ the continuity (in particular measurability) of $\ovl{\Theta}$ allows us to define $\nu\ovl{\Theta}$ by the formula:
\begin{align}\label{eq:transition}
\nu\ovl{\Theta}(h)=\int \ovl{\Theta}([N,o])(h)\,d\nu([N,o])=\int \Ex{h(N^\theta,o)}d\nu([N,o]),
\end{align}
where $h:\sC^*\to \R$ is a bounded measurable function and the expectation $\Ex{\,\cdot\,}$ is with respect to the Brownian motions.

\begin{theorem}[Hydrodynamic Limit, Proof in Section \ref{sec:proof:theo:hydlimit}]\label{theo:hydlimit} Consider a sequence \[\{N_n\}_{n\in\N}\subset \sN_f\] of finite networks where each $N_n$ has vertex set $[n]$. Make the following additional assumptions.
\begin{enumerate}
\item\label{item:lwc} {\em Local weak convergence:} $\{N_n\}_{n\in\N}$ is locally weakly convergent to a probability measure $\nu$ (in the sense of Definition \ref{def:localweak}). 
\item\label{item:nice} {\em Good limiting network:} the measure $\nu$ is supported on nice networks (cf. Definition \ref{def:goodnet}). 
\item\label{item:degree} {\em Small maximum degree:} the largest degree in $N_n$ satisfies $\max_{v\in [n]}d_v^{N_n}=n^{\eps_n}$ where $\eps_n\to 0$ as $n\to +\infty$.
\end{enumerate}

Then the sequence of random networks $\{N_n^{\theta}\}_{n\in\N}\subset \sC$ containing the random trajectories of the interacting diffusions converges to $\nu\ovl{\Theta}\in\sP(\sC^*)$ in the local weak sense and almost surely with respect to the law of the Brownian motions.\end{theorem}

\begin{remark}\label{rem:nicenet} We show in Section \ref{sec:goodexamples} that taking $(G_n)_{n\in\N}$ from the examples \ref{exe:cycle},  \ref{exe:er} and \ref{exe:degree} of almost surely (with respect to the graph randomness) local weak convergence we can construct networks $N_n$ that satisfy the assumptions of Theorem \ref{theo:hydlimit} for almost all realizations of marks and almost all realizations of the sequence $G_n$. In particular, the theorem is true for the case of $\ER$ and GW trees with probability $1$ with respect to the law of the $\ER$ graphs.
\end{remark}

The hydrodynamic limit is in general stated for the empirical measure over particle trajectories. 
\[L_n=\frac{1}{n}\sum_{v\in N_n}\delta_{\theta_v^{N_n}(\cdot)}\in \sP(C([0,T];\R)).\]

We can obtain $\theta_o^N$ as a projection of $(N^\theta,o).$  Using that the projection is continuous we obtain the following corollary.
\begin{corollary}[Proof omitted]\label{cor:omitted}In the conditions of Theorem \ref{theo:hydlimit}, $L_n$ converges almost surely to the law of $\theta_o^N$ in the weak topology of $\sP(C([0,T];\R))$.
\end{corollary}

\begin{remark} To see that the limiting object is not Markovian consider the case when $\nu$ is supported on a deterministic rooted $d-$regular tree $(T,o).$ Then the evolution $\theta_o^T(s)$ in the time interval $t\leq s\leq T$ depends also on the values $\{\theta_v^T(t):v\sim_To\}$ (see (\ref{eq:sde})). Observe however that the full vector $\vecg\theta^N=(\theta_v^N)$ is Markovian. Since it is a strong solution then $\vecg\theta^N(t+s)-\vecg\theta^N(t)$ is independent of $\sF_t=\sigma(B_v(u):0\leq u\leq t,v\in N).$
\end{remark}
As shown in \cite[Propostion~2.2]{Sznitman_Chaos}, the hydrodynamic limit immediately implies the Propagation of Chaos (cf. \cite[Definition~2.1]{Sznitman_Chaos}).
\begin{corollary}[Propagation of Chaos]\label{theo:chaos} Let $k\in \N$ and \[f_1,\cdots,f_k:\sC^*\to\R, \mbox{ such that } \blnorm{f_i}\leq 1, \forall 1\leq i\leq k.\]
Consider $(N_i,o_i)$ i.i.d. with law $\nu$. Under our assumptions,
\[\lim_{n\to\infty}\Ex{\prod_{i=1}^kU(N_n^{\theta})(f_i)}\to\prod_{i=1}^k\Ex{f_i(N_i^\theta,o_i)}.\]
\end{corollary}

\section{The Locality Lemma}\label{sec:locality}

In this section, we introduce the main technical tool in our proofs, the Locality lemma. It will be present in the proofs of all the results in Section \ref{sec:vsparse:main}.

The Locality lemma basically establishes that our interacting diffusions over a finite subset $H_0$ of vertices are indifferent to parts of the network that are far away from $H_0$.

\begin{lemma}[Locality]\label{carne:thecorollary} Consider a rooted network \[(N,o)=(G,\vecg{\mu},\vecg{\omega},\vecg{\theta}(0),o)\] according to Definition \ref{def:spaceofnetworks}. Let $H_0$ and $H$ be finite subgraphs of $G$ with $H_0\subset H$. Let $\ovl{N}$ be the sub-network induced by $H$ (cf. Section \ref{sec:networks}). Fix functions $\psi,\phi$ that are Lipschitz with $\phi$  bounded. Let $(B_v)_{v\in V}$ be i.i.d. Brownian motions associated with the vertices of $G$.

Following Definition \ref{def:supersparseID}, {\em assume} that we can define a system of interacting diffusions $\vecg{\theta}^N(\cdot)$ over $N$ from the $(B_v)_{v\in V}$ (with $\psi$ and $\phi$ fixed above and in the time interval $[0,T]$). Also build a system $\vecg{\theta}^{\overline{N}}(\cdot)$ over $\overline{N}$ with the same Brownian motions (this works because $\overline{N}$ is finite). Then the following holds: there exist $C,r_0>0$ depending only on $T$, $\pnorm{Lip}{\psi}$ and $\blnorm{\phi}$ such that, almost surely on the Brownian motion randomness, 
\[\mbox{if }r:={\rm dist}(H_0,\partial H)\geq r_0,\mbox{ then }\max_{v \in H_0}\left(\sup_{t\leq T}|\theta_v^N(t)-\theta^{\ovl{N}}_{v}(t)|\right)\leq C|\partial H|\expp{-r\log r}.\]\end{lemma}

The rest of the section will be devoted to the proof of Lemma \ref{carne:thecorollary}. However, if the reader so wishes, she or he can skip the proof and go directly to Section \ref{sec:conttheta}. In this section the only randomness comes from the diffusions. So every ``almost surely" statement in this section is with respect to the law of the Brownian motions.

The general idea of the proof is the following: In \S \ref{sec:prelimlocality} we rewrite the diffusions $\vecg{\theta}^N(\cdot)$ and $\vecg{\theta}^{\overline{N}}(\cdot)$ so that they can be compared easily. In \S \ref{sec:gronwalllocality}, we prove a linear Gronwall inequality for the difference between the two systems. The proof is finished in \S \ref{sec:finishlocality}. At this step, we need to analyze a certain matrix exponential. We do so via the theory of continuous-time random walks, most importantly Carne-Varoupoulos heat kernel bound \cite[Section~5.1]{barlow2017random}. 

As noted in the introduction, Lacker et al. \cite{Kavita1,Kavita2} essentially bypass locality estimates in their proofs. We strongly believe that the same locality estimates can be proven in their framework.

\subsection{Preliminaries}\label{sec:prelimlocality}

To avoid cumbersome notation, we adopt the following notation conventions. Objects related to the network $N=(G,\vecg{\mu},\vecg{\omega},\vecg{\theta(0)})$ are written without superscripts. Degrees of vertices are indicated via $d_v$.  We define a matrix $P$ indexed by the vertices $V$ of $N$ via: 
\[P_{uv}:= \frac{\mu_{uv}}{\mu_v}\, \Ind{\{\mu_v>0\}}.\]
With this notation, we are assuming the existence of the interacting diffusions over $N$ that may be written as 
\begin{eqnarray}\label{eq:exptheta}
{\rm d}\theta_v(t)&=&\sum_{u\in G}P_{u,v}\phi(\theta_u(t),\theta_v(t);\omega_v,\omega_u){\rm d}t+\psi(\theta_v(t),\omega_v){\rm d}t + {\rm d}B_v(t).
\end{eqnarray}

The network $\overline{N}=(H,\vecg{\mu}\mid_H,\vecg{\omega}\mid_H,\vecg{\theta(0)}\mid_H)$ is induced by the sub-graph $H\subset G$. We will write the corresponding process somewhat differently. Define $\overline{\mu}_{uv}:=\mu_{uv}\Ind{\{v,u\in H\}}$ and $\overline{\mu}_v:=\sum_{u}\overline{\mu}_{uv}$. We set:
\[\overline{P}_{uv}:= \frac{\overline{\mu}_{uv}}{\overline{\mu}_v}\, \Ind{\{\overline{\mu}_v>0\}}.\]

This matrix is in general different from $P$ if $v$ or $u$ are either outside of $H$ or in $\partial H$. 

We may define another system of diffusions satisfying:
\begin{eqnarray}\label{eq:expthetabar}
{\rm d}\overline{\theta}_v(t)&=&\sum_{u\in G}\overline{P}_{u,v}\phi(\overline{\theta}_u(t),\overline{\theta}_v(t);\omega_v,\omega_u){\rm d}t+\psi(\overline{\theta}_v(t),\omega_v){\rm d}t + {\rm d}B_v(t)
\end{eqnarray}
for each $v\in V$ (and not just the vertices in $H$), with initial conditions $\overline{\theta}_v(0) = \theta_v(0)$. With this definition, the diffusions inside $H$ do not interact with those outside $H$. Since $H$ is finite, the system inside $H$ has a unique strong solution, so the $\overline{\theta}_v(\cdot)$ with $v\in H$ correspond exactly to the $\theta^{\overline{N}}_v(\cdot)$ in the statement of the Lemma.  

Our goal then is to bound $\sup_{t\leq T}|\theta_v(t) - \overline{\theta}_v(t)|$ for $v\in H_0$. More specifically, it suffices to prove the next result
\begin{lemma}\label{le:carne3} Let $v \in H.$  Then almost surely
\[\mbox{ if }{\rm dist}(v,\partial H)=:r\geq r_0, \mbox{ then }\sup_{t\leq T}\left(\theta_v(t)-\tH{v}(t)\right)\leq C|\partial H|\expp{-r\log r}.\]
\end{lemma}

Indeed, once we have this, Lemma \ref{carne:thecorollary} follows if we take the supremum over $v\in H_0.$ The remainder of this section is dedicated to the proof of Lemma \ref{le:carne3}. 

\subsection{A Gronwall bound}\label{sec:gronwalllocality}

The next proposition will be used in our Grownwall argument. To state it, we let $\til{H}$ denote the set of vertices within distance at most $1$ of $H$, that is \[\til{H}=\{u\in G: u\in H \mbox{ or } \exists\, v\in H \mbox{ with } v\sim u\}.\] Also, we write $I$ for the identity matrix.

\begin{proposition}\label{prop:carne1} Let $v \in H$. The following inequality holds almost surely for all $t \in [0,T],$
\begin{eqnarray}
|\theta_v(t)-\ovl{\theta}_v(t)|\leq C\int_{0}^{t}\sum_{u\in \til{H}}(\oP+I)_{uv}|\theta_u(s)-\ovl{\theta}_u(s)|ds +2CT\Ind{\{v\in\partial H\}}.\nonumber 
\end{eqnarray}
where $C>0$ depends only on $\pnorm{Lip}{\psi}$ and $\blnorm{\phi}$.\end{proposition}

\begin{proof}Since the initial conditions are coupled to be equal they cancel when we calculate $\theta_v(t)-\tH{v}(t).$ The Brownian motion also cancels. From Equation \ref{eq:exptheta} we obtain

\begin{eqnarray}
\theta_v(t)-\ovl{\theta}_v(t)=\int_{0}^{t} \Delta_1(v,s)+\Delta_2(v,s)+\Delta_3(v,s) {\rm d}s \nonumber
\end{eqnarray}
where
\begin{eqnarray}
\Delta_1(v,s)&:=&\sum_{u\in G}(P_{uv}-\oP_{uv})\phi(\theta_u(s),\theta_v(s);\omega_u,\omega_v),\nonumber \\
\Delta_2(v,s)&:=&\sum_{u\in G}\oP_{uv}(\phi(\theta_u(s),\theta_v(s);\omega_u,\omega_v)-\phi(\tH{u}(s),\tH{v}(s);\omega_u,\omega_v)),\nonumber \\
\Delta_3(v,s)&:=&\psi(\theta_v(s);\omega_v)-\psi(\tH{v}(s);\omega_v).\nonumber
\end{eqnarray}
First observe that the sums involving $P$ and $\oP$ for $u \in G$ are in fact over $u\in \til{H}$ since $v\in H.$ This is due the fact this matrices have non-zero entries $P_{\cdot,v},\oP_{\cdot,v}$ just for neighbors of $v$.
In this way we have the following bounds
\begin{eqnarray}
|\Delta_1(v,s)|&\leq &\supnorm{\phi}\left(\sum_{u\in \til{H}}|P_{uv}-\oP_{uv}|\right),\nonumber \\
|\Delta_2(v,s)|&\leq &\sum_{u\in \til{H}}\oP_{uv}\pnorm{Lip}{\phi}\left(|\theta_v(s)-\ovl{\theta}_v(s)|+|\theta_u(s)-\ovl{\theta}_u(s)|\right)\, \mbox{, and} \nonumber \\
|\Delta_3(v,s)|&\leq &\pnorm{Lip}{\psi}|\theta_v(s)-\ovl{\theta}_v(s)|.\nonumber
\end{eqnarray}

Now we bound the RHS of the first inequality. For this, we note that if $v\not \in \partial H$ or $\mu_v=0$, then $P_{uv}=\overline{P}_{uv}$ for all $u$. On the other hand, if $v\in \partial H$ and $\mu_v>0$, then:

\begin{eqnarray}
\sum_{u\in\til{H}}|P_{uv}-\oP_{uv}|&=&\sum_{u \in H}|P_{uv}-\oP_{uv}|+\sum_{u \notin H}|P_{uv}-\oP_{uv}| \nonumber \\
(u\notin H \implies \oP_{uv}=0)&=&\sum_{u \in H}|P_{uv}-\oP_{uv}|+\sum_{u \notin H}|P_{uv}| \nonumber \\
(u\in H\implies \mu_{uv}=\ovl{\mu}_{uv})&=&\sum_{u \in H}\mu_{uv}\modulo{\dfrac{1}{\mu_v}-\dfrac{1}{\ovl{\mu}_{v}}}+\sum_{u \notin H}|P_{uv}| \nonumber \\
&\leq & \ovl{\mu}_{v}\modulo{\dfrac{1}{\mu_v}-\dfrac{1}{\ovl{\mu}_{v}}}+(\mu_v-\ovl{\mu}_{v})\dfrac{1}{\mu_v} \nonumber \\
(\ovl{\mu}_v\leq \mu_v)&=&2\left(1-\dfrac{\ovl{\mu}_{v}}{\mu_{v}}\right)\leq 2.\nonumber
\end{eqnarray}  

In any case we have that,
\[\sum_{u\in \til{H}}|P_{uv}-\oP_{uv}|\leq 2\Ind{v \in \partial H}.\]
Combining these bounds, we obtain the result.\end{proof}

\subsection{End of proof}\label{sec:finishlocality}

We now finish the proof of Lemma \ref{le:carne3}, which implies Lemma \ref{carne:thecorollary}. We will apply the Linear Gronwall's Inequality (Corollary \ref{gr:3rd}) to finish the proof of Lemma \ref{le:carne3}. 

Going back to Proposition \ref{prop:carne1} we observe that 
$\til{H}$ is finite. In particular, the matrices $\overline{P}$ and $I$ considered are finite-dimensional. So we can apply Corollary \ref{gr:3rd} with:
\begin{enumerate}
\item $\vec{u}(t)=(|\theta_v(t)-\ovl{\theta}_v(t)|)_{v \in \til{H}},$ which have continuous entries,
\item $\vec{a}(t)=\vec{a}:=\left(CT\Ind{[v \in \partial H]}\right)_{v \in \til{H}},$ and we observe that each entry of this vector is non-negative, and
\item $M(t)_{uv} = M_{uv}=C(\oP+I)_{uv},$ for $u,v\in\til{H}$ and we observe that this matrix does not depend on time $t$, it is entry-wise non-negative and it is finite dimensional.
\end{enumerate}

In vector notation, we obtain:
\[\vec{u}(t)\leq \int_0^t (M\vec{u}(s) + \vec{a})\,ds\]
and the Corollary says that $\vec{u}(t)\leq  \exp(tM)\,\vec{a}$ entrywise. This is the same as saying that, for each $v\in H$
\[|\theta_v(t) - \overline{\theta}_v(t)|\leq CT\,\sum_{u\in \partial H}\exp(Ct\,(\oP+I))_{uv} \leq CTe^{2T}\,\sum_{u\in \partial H}\exp(Ct\,(\oP-I))_{uv}.\]
To bound this last expression, we note that 
\[\exp(Ct\,(\oP-I))_{uv} = \overline{\mu}_u\,q_{Ct}(u,v)\]
where $q_{Ct}(u,v)$ is the heat kernel at time $Ct$ of a continuous time random walk over $H$ with transition rates equal to $1$ and reversible transition probabilities $\overline{P}_{uv}$ (reversibility follows from symmetry of $\mu_{uv}$). The Carne-Varoupoulos bound for the heat kernel (Theorem 5.17 of \cite[Section~5.1]{barlow2017random}) implies that for any time $s\geq 0$ and any $v,u \in H$ with $R={\rm dist}(v,u)\geq es$ ($e$ is the Euler constant)
\[q_s(v,u)\leq \dfrac{1}{\ovl{\mu}_{v}\vee\ovl{\mu}_{u}}\expp{-s-R\log \dfrac{R}{es}}\leq \dfrac{1}{\ovl{\mu}_{u}}\expp{-s-R\log \dfrac{R}{es}}.\]
We apply this with $s = Ct\leq CT$, and obtain that, if $R={\rm dist}(v,\partial H)\geq eCT$, then:
\[\sup_{t\leq T}|\theta_v(t) - \overline{\theta}_v(t)|\leq |\partial H| CTe^{2T}\,\expp{-Ct-R\log \dfrac{R}{eCT}}.\]

So we finish by taking $r_0=\lceil eCT\rceil$ and adjusting $C$ accordingly.

\section{Interacting diffusions over infinite graphs}\label{sec:conttheta}

In this section, we prove Theorem \ref{theo:exttheta}. We first give a sketch of the argument.

\begin{enumerate}
\item Our main task will be to construct our system of diffusions over infinite networks via limits of systems over finite sub-networks. More precisely, for an infinite rooted network $(N,o)$ and a radius $r\geq 1$ we have the finite rooted network $(N,o)_r$ (cf. Section \ref{sec:networks}). It is clear that the sequence $((N,o)_r)_{r\geq 1}$ converges in the local topology towards to $(N,o)$.
 
\item The random networks $(N,o)_r^{\theta}$ ($(N,o)_r^{\theta}\neq (N^\theta,o)_r$, see Definition \ref{def:Ntheta}) are well defined since the networks $(N,o)_r$ are finite. We will use the fact that $(N,o)$ is nice (cf. Definition \ref{def:goodnet}) and Lemma \ref{carne:thecorollary} to show that the sequence $(N,o)_r^{\theta}$ (considering $o$ as the root) has a limit in distribution. We will then show that the limit is the rooted network $(N^{\theta},o)$ replacing the initial conditions by $\vecg{\theta}^N(\cdot)$ that is the unique strong solution of the infinite system of SDEs in (\ref{eq:sde}) for $N$. 

\item To prove the continuity of the map $(N,o) \mbox{ nice}\mapsto \mbox{Law of } (N^\theta,o)$ we will also need to use that for finite networks the map $(N,o) \mbox{ finite}\mapsto \mbox{Law of } (N^\theta,o)$ is continuous (Lemma \ref{le:samegraph}).
\end{enumerate}

We give the formal proof of Theorem \ref{theo:exttheta} over the next subsections.

\subsection{Proof of existence}\label{sub:prooftheo}

Throughout the proof, we will assume that we are given i.i.d. standard Brownian motions $(B_v)_{v\in V}$ defined on the vertices of our network. The equalities and estimates in all Section \ref{sec:conttheta} will hold almost surely with respect to the law of the Brownian motions. For each $r\geq 0$, we have a finite network $(N,o)_r$ with Brownian motions attached to its vertices. Following Definition \ref{def:supersparseID}, we can build a system of interacting diffusions over the vertices $v\in (N,o)_r$ via:
\begin{eqnarray}\label{eq:forfixedr}
{\rm d}\theta^{(r)}_{v}(t)=\frac{1}{\mu^{(r)}_v} \sum_{u \in (N,o)_r}\mu_{uv}\phi(\theta^{(r)}_{u}(t),\theta^{(r)}_{v}(t);\omega_v,\omega_u){\rm d}t+\psi(\theta^{(r)}_{v}(t);\omega_v){\rm d}t + {\rm d}B_v(t),
\end{eqnarray}
with the weights $\vecg{\mu}=(\mu_{vu})_{vu\in E_G}$, media variables $\vecg{\omega}=(\omega_v)_{v\in V_G}$ and initial conditions $\vecg{\theta^{(r)}}(0)=(\theta_v(0))_{v\in V_G}$ determined by the network $(N,o)_r$. We also use the notation $\mu^{(r)}_v := \sum_{u \in (N,o)_r}\mu_{uv}$. The point of this construction is that it couples our interacting diffusions over $(N,o)_r$ for all $r$ simultaneously. 

We now apply Lemma \ref{carne:thecorollary}, with $H_0=(N,o)_s$ for some $s\geq 0$, $H=(N,o)_{r}$ for $r\geq s$, and with the network $N$ replaced by a finite ball $(N,o)_{r'}$ with $r'\geq r$. Since $(N,o)_{r'}$ is always finite, the solution to the system is well defined and we can indeed apply the Lemma. We have that ${\rm dist}((N,o)_s,\partial (N,o)_r)=r-s.$ Moreover, since $(N,o)$ is nice, there exist $a>0$ independent of $r$ with $|\partial(N,o)_r|\leq ae^{ar}$. We conclude that 
\[\forall r'\geq r\geq s+r_0,\ \max_{v \in (N,o)_s}\left(\sup_{t\leq T}|\theta_v^{(r)}(t)-\theta^{(r')}_{v}(t)|\right)\leq Ca\exp(ar-(r-s)\log (r-s)).\]

The fact that the RHS goes to $0$ for fixed $s$ and $r',r\to +\infty$ implies that, for each $v\in V$, $\theta_v^{(r)}(\cdot)$ is a Cauchy sequence (in the uniform norm over $[0,T]$) and converges over $[0,T]$ to a continuous function $\theta^N_v(\cdot)$. We also have the estimate:
\begin{equation}\label{eq:convergiu}\forall r\geq s+r_0,\,\max_{v \in (N,o)_s}\left(\sup_{t\leq T}|\theta_v^{(r)}(t)-\theta^N_{v}(t)|\right)\leq Ca\expp{ar-(r-s)\log (r-s)},\end{equation}
and we will show now that this estimate implies that $\vecg{\theta}^N=(\theta^N_v(\cdot))_{v\in V}$ is a system of interacting diffusions over $N$ in the sense of Definition \ref{def:supersparseID}. To do this first observe that $\mu_v^{(r)}=\mu_v$ for large $r$. For any fixed $r\geq 1$ we have from \eqref{eq:forfixedr}
\begin{multline}\label{eq:integralform}
\begin{split}
\theta^{(r)}_{v}(t)=\theta^{(r)}_{v}(0)+\frac{1}{\mu^{(r)}_v} \sum_{u \in (N,o)_r}\mu_{uv}\int_0^t\phi(\theta^{(r)}_{u}(s),\theta^{(r)}_{v}(s);\omega_v,\omega_u){\rm d}s+\\
+\int_0^t\psi(\theta^{(r)}_{v}(s);\omega_v){\rm d}s + B_v(s).
\end{split}
\end{multline}
The estimate \eqref{eq:convergiu} says that for all $u$ neighbour of $v$ it holds that $\theta^{(r)}_u$ converges to $\theta^N_u$ in the supremum norm and almost surely. Using the basic fact that if $f_n$ converges in the supremum norm to $f$ then $\int_0^tf_n(s){\rm d}s\to\int_0^tf(s){\rm d}s$ we obtain from \eqref{eq:integralform} that the following holds almost surely
\begin{multline*}
\begin{split}
\theta^{N}_{v}(t)=\theta^{N}_{v}(0)+\frac{1}{\mu^{N}_v} \sum_{u \in N}\mu_{uv}\int_0^t\phi(\theta^{N}_{u}(s),\theta^{N}_{v}(s);\omega_v,\omega_u){\rm d}s+\\
+\int_0^t\psi(\theta^{N}_{v}(s);\omega_v){\rm d}s + B_v(s),
\end{split}
\end{multline*}
as we wanted.
\subsection{Proof of uniqueness}

The above implies {\em existence} of a system of interacting diffusions over $N$. {\em Uniqueness} of such a process is also easy to obtain. Indeed, suppose $\vecg{\beta}^N$ is another strong solution to the same system of equations defined in terms of the same Brownian motions $(B_v)_{v\in V}$. Then an application of the locality result, Lemma \ref{carne:thecorollary}, to $H_0=\{v\}$, $H=(N,o)_r$ for large $r$, and $N$, reveals that \eqnref{convergiu} must also hold with $\beta^N_v$ replacing $\theta^N_v$. Therefore,
\[\forall v\in V\,:\,\beta^N_v(\cdot)=\lim_{r\to +\infty}\theta_v^{(r)}(\cdot) = \theta^N_v(\cdot).\] 

\subsection{Proof of continuity}

What we have seen so far is that for each nice rooted network $(N,o)$ one may uniquely define a system of interacting diffusions $\vecg{\theta}^N(\cdot)$. Let $[N^{\theta},o]\in \sC^*$ be the resulting random network when one replaces the initial condition for the diffusions as new marks to the vertices (in the sense of Definition \ref{def:Ntheta}). So let $\ovl{\Theta}[N,o]$ denote the law of $[N^{\theta},o]\in \sC^*$ (Definition \ref{def:Theta}). The uniqueness statement above implies that $\ovl{\Theta}$ extends the definition of $\Theta$ over finite networks. 

We must now show that $\ovl{\Theta}$ is a {\em continuous} map from $\sB^*$ (the set of nice networks) to $\sP(\sC^*)$ (the set of probability measures over $\sC^*$ with the BL metric). We start with some preliminaries. We note once again that, due to Lemma \ref{carne:thecorollary}, we have the  more precise estimate:
\begin{equation}\label{eq:boafronteira}\forall r\geq s+r_0,\,\max_{v \in (N,o)_s}\left(\sup_{t\leq T}|\theta_v^{(r)}(t)-\theta^N_{v}(t)|\right)\leq C|\partial(N,o)_r|\expp{-(r-s)\log (r-s)}.\end{equation}

This bound immediately translates to a bound for the distance between $(N^{\theta},o)$ and $(N,o)_r^{\theta}$ in the space $\sC^*.$
Furthermore, this construction of $(N^{\theta},o)$, and $(N,o)_r^{\theta}$ is a coupling of the measures $\ovl{\Theta}[N,o]$ and $\ovl{\Theta}[N,o]_r$, and we obtain:
\begin{equation}\label{eq:taquasecontinuity}\forall r\geq s+r_0,\ d_{BL}(\ovl{\Theta}[N,o],\ovl{\Theta}[N,o]_r)\leq \frac{1}{1+s} +C|\partial(N,o)_r|\expp{-(r-s)\log (r-s)}.\end{equation}

The important point is that \eqnref{taquasecontinuity} is applicable to {\em all} $[N,o]\in\sB^*$.

Now consider  a sequence $[N_n,o_n]$  of networks converging to $[N,o]$. We wish to show that $d_{BL}(\ovl{\Theta}[N,o],\ovl{\Theta}[N_n,o_n])\to 0$. To do this, we use the following observation: for any fixed $r$, the probability laws corresponding to $[N_n,o_n]_r$ converge to those of $[N,o]_r$:
\begin{align}\label{eq:lim-finitenet}
d_{BL}(\ovl{\Theta}[N,o]_r,\ovl{\Theta}[N_n,o_n]_r)\to 0.
\end{align}

Indeed, this is true because $[N,o]_r$ and $[N_n,o_n]_r$  are {\em finite}. This is discussed in detail in Section \ref{sec:limfinitenet} of the Appendix.

The triangle inequality gives:
\begin{eqnarray*}d_{BL}(\ovl{\Theta}[N,o],\ovl{\Theta}[N_n,o_n])&\leq  & d_{BL}(\ovl{\Theta}[N,o]_r,\ovl{\Theta}[N_n,o_n]_r)\\ & & + d_{BL}(\ovl{\Theta}[N,o]_r,\ovl{\Theta}[N,o]) \\ & & +d_{BL}(\ovl{\Theta}[N_n,o_n],\ovl{\Theta}[N_n,o_n]_r).\end{eqnarray*}
When $n\to +\infty$, the first term in the RHS shrinks to $0$. Using \eqnref{taquasecontinuity} to bound the other two terms, we obtain:
\begin{multline*} \limsup_{n\to\infty}d_{BL}(\ovl{\Theta}[N,o],\ovl{\Theta}[N_n,o_n])\leq  \frac{1}{1+s} +C|\partial(N,o)_r|\expp{-(r-s)\log (r-s)}\\   +  \limsup_{n\to\infty}\left(\frac{1}{1+s} +C|\partial(N_n,o_n)_r|\expp{-(r-s)\log (r-s)}\right). 
\end{multline*}
The local convergence of $[N_n,o_n]$ to $[N,o]$ implies that $|\partial(N_n,o_n)_r|\to |\partial(N,o)_r|$, so:
\[\limsup_{n\to\infty} d_{BL}(\ovl{\Theta}[N,o],\ovl{\Theta}[N_n,o_n])\leq \frac{2}{1+s} +2C|\partial(N,o)_r|\expp{-(r-s)\log (r-s)}.\]
Since $r\geq r_0+s$ are arbitrary (and $r_0$ is constant), we may let $s=r/2$ and make $r\to +\infty$ to obtain
\[\limsup_{n\to\infty} d_{BL}(\ovl{\Theta}[N,o],\ovl{\Theta}[N_n,o_n])\leq 0,\]
as desired.

\section{Hydrodynamic limit}\label{sec:proof:theo:hydlimit}
In this section we will prove Theorem \ref{theo:hydlimit}.

Our goal is the following. Given that $U(N_n)\to\nu$ in the space $\sP(\sN^*)$ and up to some additional conditions we want to show that $U(N_n^{\theta})\to \nu\ovl{\Theta}$ in $\sP(\sC^*)$ almost surely. In this section ``almost surely" means ``almost surely with respect to the law of the Brownian motions". By a standard argument, it suffices to show that for any test function $h:\sC^*\to\R$ with $\blnorm{h}\leq 1$, 
\[\mbox{{\bf Goal:} }U(N_n^{\theta})(h)\to \nu\ovl{\Theta}(h)\mbox{ almost surely,}\]
with respect to the law of the Brownian motions.
It will be useful to consider the intermediate expression:
\[U(N_n)\ovl{\Theta}(h)=\Ex{U(N_n^{\theta})(h)}\]
where the expectation $\Ex{\,\cdot\,}$ is with respect to the Brownian motions. Since $\ovl{\Theta}:\sB^*\to\sP(\sC^*)$ is continuous (by Theorem \ref{theo:exttheta}) and $U(N_n)\to \nu$ weakly, one may easily show that
\begin{align}\label{eq:convinmean}
U(N_n)\ovl{\Theta}(h)\to \nu\ovl{\Theta}(h).
\end{align}

Therefore, our goal is tantamount to showing that:
\begin{equation}\label{eq:goalhydro}\mbox{{\bf Goal (restated):} }U(N_n^{\theta})(h) - \Ex{U(N_n^{\theta})(h)}\to 0\mbox{ almost surely,}\end{equation}
with respect to the law of the Brownian motions, where we recall $\|h\|_{BL}\leq 1$. The proof idea for \eqnref{goalhydro} is to use the fact that $U(N_n^{\theta})(h)$ is a function of independent Brownian motions. If we could control the effect of replacing one of the Brownian motions, then we can prove concentration by Azuma's inequality. To make this work, we will need to consider a truncated process for $r$ fixed given by the networks: 
\[(N_n(v),v)_r^{\theta}\]
replacing the initial conditions for the interacting diffusions $\vecg{\theta}^{(N_n(v),v)_r}(\cdot)$. With this motivation we define the $r-$neighborhood empirical measure
\begin{align}\label{eq:rneigh}
U^{(r)}(N_n^{\theta})=\frac{1}{n}\sum_{v=1}^n\delta_{(N_n(v),v)_r^{\theta}}
\end{align}
After considering these networks for $r$ fixed we will need to return to our original network. In what follows we will need to use Azuma's inequality.
\begin{theorem}[Azuma's inequality, Theorem 6.2, \cite{boucheron2013concentration}]\label{thm:azuma}
Let $\sX$ be a measurable space. Assume that the function $f:\sX^n\to\R$ satisfies the bounded differences assumption: for each $1\leq i \leq n,$
\[\sup_{\substack{x_1,\cdots,x_n\in \sX\\x_i'\in \sX}}|f(x_1,\cdots,x_{i-1},x_i,x_{i+1},\cdots,x_n)-f(x_1,\cdots,x_{i-1},x_i',x_{i+1},\cdots,x_n)|\leq c_i\]
and define $\til{\nu}=\frac{1}{4}\sum_{i=1}^nc_i^2.$ Let $Z = f(X_1 ,...,X_n )$ where the $X_i\in \sX$ are independent. Then
\[\Pr{Z-\Ex{Z} > t}\leq e^{-t^2/(2\til{\nu})}.\]
\end{theorem}

\begin{lemma}\label{le:concentration} For any $r\in\N$ fixed, the following holds almost surely, with respect to the law of the Brownian motions,
\[\limsup_{n\to\infty}\left|U^{(r)}(N_n^{\theta})(h)-\Ex{U^{(r)}(N_n^{\theta})(h)}\right|=0.\]
\end{lemma}
\begin{proof} Recall that the assumptions on our networks imply:
\begin{align}\label{eq:degree}
\max_{v\in [n]}d_v^{N_n}=n^{\eps_n} \mbox{ with } \lim_{n\to\infty}\eps_n=0.
\end{align}
From (\ref{eq:rneigh}) we have that
\begin{flalign*} 
U^{(r)}(N_n^{\theta})(h)=\frac{1}{n}\sum_{v=1}^nh[(N_n(v),v)_r^\theta],
\end{flalign*}
and the randomness of the rooted network $(N_n(v),v)_r^\theta$ is  the random vector \[(\theta_u^{(N_n(v),v)_r}(\cdot))_{u\in (N_n(v),v)_r},\mbox{ for $v\in [n].$}\] 

The strong solution assumption implies that for each $v\in [n]$ there exists a measurable function \[g_v:C([0,T];\R)^{|(N_n(v),v)_r|}\to \sC^* \mbox{ such that}\] 
\[g_v\left((B_z(\cdot))_{z\in (N_n(v),v)_r}\right)=(N_n(v),v)_r^\theta.\]

Therefore,
\begin{flalign*} 
U^{(r)}(N_n^{\theta})(h)&=\frac{1}{n}\sum_{v=1}^nh[(N_n(v),v)_r^\theta]\\
&=\frac{1}{n}\sum_{v=1}^nh\left(g_v\left((B_z(\cdot))_{z\in (N_n(v),v)_r}\right)\right)=:f(B_1(\cdot),\cdots,B_n(\cdot)).
\end{flalign*}

Now fix a vertex $w\in [n]$ and suppose we change the function $B_w(\cdot).$ Then  whenever $w\notin (N_n(v),v)_{r}$ the function $g_v\left((B_z(\cdot))_{z\in (N_n(v),v)_r}\right)$ is unchanged. That is, the only functions $g_v$ that are changed are those with $v \in (N_n(w),w)_{r}.$ Using that $\supnorm{h}\leq 1$ we can conclude that $f$ satisfy the bounded difference inequality (Theorem \ref{thm:azuma}) with \[c_w=\frac{|(N_n(w),w)_r|}{n}.\] From (\ref{eq:degree}) the size of $(N_n(w),w)_r$ is bounded by $n^{(r+1)\eps_n}$ (to see this write the ball as the union of spheres). In this way,
\[\til{\nu}:=\dfrac{1}{4}\sum_{w=1}^nc_w^2\leq \dfrac{1}{4}\sum_{w=1}^n\left(\frac{n^{(r+1)\eps_n}}{n}\right)^2=\frac{1}{4}\frac{n^{2(r+1)\eps_n}}{n}.\]

Theorem \ref{thm:azuma} applied twice implies
\[\Pr{\left|U^{(r)}(N_n^{\theta})(h)-\Ex{U^{(r)}(N_n^{\theta})(h)}\right|>t}\leq 2\expp{-\frac{2t^2n}{n^{2(r+1)\eps_n}}}.\]
This bound is summable in $n$ for any $r\geq 1$ and $t>0$ fixed because $\eps_n\to 0$. Therefore, Borel-Cantelli Lemma finishes the proof. \end{proof}

To continue, we must compare the truncated network in this Lemma with the original $N_n$. We first bound:
\[U(N_n^{\theta})(h) - U^{(r)}(N_n^{\theta})(h),\]
which is an average of differences:
\[h([N_n^{\theta}(v),v]) - h([N_n(v),v]_r^{\theta})\]
over $v\in [n]$. Since $\|h\|_{BL}\leq 1$, we have:
\[|h([N_n^{\theta}(v),v]) - h([N_n(v),v]_r^{\theta})|\leq \dist{\sC^*}{[N_n^{\theta}(v),v]}{[N_n(v),v]_r^{\theta}}\wedge 2.\]
Let $r_0$ be the constant of Lemma \ref{carne:thecorollary}. From this Lemma, we have the bound:
\begin{equation}\label{eq:carneapplied}
\dist{\sC^*}{[N_n^{\theta}(v),v]}{[N_n(v),v]_r^{\theta}}\leq \frac{1}{1+s}+C|\partial(N_n(v),v)_r|\exp(-(r-s)\log (r-s))
\end{equation} 
whenever $r\geq s+r_0.$
Averaging over $v\in[n]$, we conclude
\begin{equation}\label{eq:distruncated}|U^{(r)}(N_n^{\theta})(h)-U(N_n^{\theta})(h)|\leq \frac{1}{n}\sum_{v\in[n]}\left(\frac{1}{1+s} + C|\partial(N_n,v)_r|\,e^{-(r-s)\log (r-s)}\right)\wedge 2\end{equation}
for all $r\geq s+r_0$. 

We know $N_n$ converges in the local weak sense to $\nu$. The measure $\nu$ is supported on nice networks. That is, we have that $\nu(\cup_{a=1}^{+\infty}\sB^*_a)=1$ with
\begin{align}\label{eq:morenice}\sB^*_{a}:=\{(N,o)\in\sB^*: |\partial (N,o)_r|\leq ae^{ar},\forall r\geq 1\}.\end{align}
Since the $\sB_a^*$ are increasing, we have $\nu(\sB^*_a)\to 1$ as $a\to +\infty$. Also, by local weak convergence:
\[\frac{1}{n}\sum_{u\in[n]}\Ind{\{|\partial(N_n,u)_r|>ae^{ar}\}}\rightarrow \nu(\{([N,o]\,:\,|\partial(N,o)_r|>ae^{ar}\})= 1 - \nu(\sB^*_a).\]
We bound the terms in the RHS of \eqnref{distruncated} according to whether $|\partial(N_n,v)_r|$ is either bounded by $ae^{ar}$ or not; in the latter case the terms are simply bounded by $2$. We deduce that, when $n\to +\infty$:
\begin{eqnarray}\label{eq:distruncated2}|U^{(r)}(N_n^{\theta})(h)-U(N_n^{\theta})(h)| 
& \leq& \frac{1}{1+s}+Cae^{ar-(r-s)\log (r-s)} \\ \nonumber & & +\frac{2}{n}\sum_{u\in[n]}\Ind{\{|\partial(N_n,u)_r|>ae^{ar}\}}\\ \nonumber
&\to &\frac{1}{1+s}+Cae^{ar-(r-s)\log (r-s)} + 2(1 - \nu(\sB^*_a)).\end{eqnarray}
The same bound holds for the difference of the expectations \[|\Ex{U^{(r)}(N_n^{\theta})(h)-U(N_n^{\theta})(h)}|.\] Now, for each fixed $r,s,a$ as above:
\begin{eqnarray*}|U(N_n^{\theta})(h) - \Ex{U(N_n^{\theta})(h)}| &\leq & |U(N_n^{\theta})(h) - U^{(r)}(N_n^{\theta})(h)| \\ & & + |\Ex{U^{(r)}(N_n^{\theta})(h)-U(N_n^{\theta})(h)}| \\ & & + |U^{(r)}(N_n^{\theta})(h)-\Ex{U^{(r)}(N_n^{\theta})(h)}|.\end{eqnarray*}
When $n\to +\infty$, the first two terms in the RHS may be controlled via \eqnref{distruncated2}. The last term goes to $0$ by Lemma \ref{le:concentration}. We obtain:
\[\limsup_{n\to\infty}|U(N_n^{\theta})(h) - \Ex{U(N_n^{\theta})(h)}|\leq \frac{2}{1+s}+2Cae^{ar-(r-s)\log (r-s)} + 2(1 - \nu(\sB^*_a)).\]
We may now let $r\to +\infty$, then $s\to +\infty$ and $a\to+\infty$ (in this order) to obtain:  
\[\limsup_{n\to\infty}|U(N_n^{\theta})(h) - \Ex{U(N_n^{\theta})(h)}| \leq 0.\]
This implies our goal \eqnref{goalhydro} and finishes the proof of the hydrodynamic limit.






\section{Numerical Simulations for the Stochastic Kuramoto Model on Galton-Watson Trees}
\label{numeric}

\subsection{Introduction}

In this section, we explain how we performed the numerical simulations for the  stochastic Kuramoto model on  Galton- Watson (GW) trees. The equations are given by $$d\theta_i = \left (\omega_i + K \sum_{j=1}^N a_{ij} \sin(\theta_j - \theta_i)  \right)dt + \varepsilon dW^i_t.$$ For this system, $\theta_j(t)$ and $\omega_j$ represent the angular phase and natural frequency of the oscillator indexed by  $j \in \{1,2, \hdots, N\}$.  The parameter $K \in \mathbb{R_+}$ represents the coupling strength between nodes, and $a_{ij} =1$ if nodes $i$ and $j$ are connected or $a_{ij} =0$ otherwise. We assume that $\{W^j_t\}_{t\geq 0}$ are independent brownian motions for each node $j$, while the noise intensity is given by $\varepsilon>0$. In our simulations,  we chose two different models for generating the GW trees:

\begin{enumerate}
\item \textbf{Binomial model:} The offspring is a binomial random variable  with distribution $Bin(n,p)$.
\item \textbf{D-Regular model:} The root node has $\mathcal{C}$ children, while the other ones have exactly $\mathcal{C} -1 $ children.
\end{enumerate}

In what follows, we fixed  $n=3$ and $p = \frac{2}{3}$ for the Binomial model (mean offspring equals to 2) and  $\mathcal{C}=3$ for the D-Regular model. We also considered the time interval  $t=[0,100]$  (arbitrary units) with step $\Delta t = 0.01$, and  noise intensity $\varepsilon = 0.05$.

\subsubsection*{The Synchronization Level}

We defined the synchronization level $Sync = Sync(h,K)$ between the root and those nodes at a  distance $\emph{h}$,  given the coupling strength $K$.  If $N$ denotes the total number of nodes and $j =1$ is defined as the root index, we define the set $$ \mathcal{D}(h) = \left\{ j \in \{2,...,N\} | \text{ node  of index $j$ is at distance $h$ from the root} \right\}   $$ and the order parameter $$ r_{h}(t) = \left \vert \frac{e^{i \theta_1 (t)}+ \sum_{j \in \mathcal{D}(h)}^{} e^{i\theta_j (t)}}{1+ \#\mathcal{D}(h) } \right \vert$$ where $\#$ denotes set cardinality.

  Our variable $Sync$ is then defined as a time average of the last $5\%$ values assumed by $r_{h}(t)$. More precisely, if we have a total of $t_n$,  the set of last  $5\%$ time indexes is given by $\mathcal{J} = \left\{ \floor{0.95t_n},...,t_n\right\} .$ Therefore, we define
  $$Sync : = \frac{\sum_{j \in \mathcal{J}} r_h(t_j)}{\#\mathcal{J}} $$ as our synchronization level parameter.

\subsection{Numerical  Simulations}

We present the details of our numerical simulations, considering the Binomial and D-Regular models. In both cases, our goal is to calculate an average synchronization parameter $\langle Sync \rangle$ between the root and those nodes at different distance values $h$, for distinct coupling strength ($K$) values. For all simulations, we considered GW trees with 13 generations ($h  \in \{1,2,\cdots,13\}$), and $K \in [0,10]$ with steps $\Delta K = 0.2$. We performed all numerical solutions with the classic Euler-Maruyama scheme for stochastic differential equations. In what follows, we describe the step-by-step algorithms for each GW model.

\begin{itemize}
\item \textbf{Binomial model.} We produce a total of  100 simulations. In each of them, we generate a GW tree with $N$ nodes and define the natural frequencies $\omega_i$, as well as the initial conditions $\theta_i(0)$ ($i \in{1,2,\hdots,N}$). We then simulate the stochastic Kuramoto model with the Euler-Maruyama scheme, computing $Sync(h,K)$ for the chosen $h$ and $K$ values (see above).  Finally, we average the synchronization levels  across the 10  simulations, obtaining  $\langle Sync \rangle (h,K)$.

\item \textbf{D-Regular model.} In this case, the GW is not random. Then we produce a total of  100 simulations, only re-sampling the initial conditions $\theta_i(0)$ and natural frequencies $\omega_i$ ($i \in{1,2,\hdots,N}$). For each simulation, we compute the synchronization level ($Sync(h,K)$) for the selected $K$ and $h$ values. We then compute the average synchronization level $\langle Sync \rangle (h,K)$ across the 100 simulations. 
\end{itemize}

\subsection{Results}

We present our  results in   Fig. \ref{Big} \textbf{a},\textbf{b}. We compared two possible cases for the initial phases:   $\theta_i(0) = 0$  or  $\theta_i(0)$ uniformly distributed on $[0,2\pi]$, for $i \in \{1,2,\hdots,N\}$. For the natural frequencies, we chose $\omega_i = 1$, or $\omega_i$ normally distributed with mean and variance equal to 1, or  $\omega_i$  uniformly distributed on $[0,2\pi]$,  for $i \in \{1,2,\hdots,N\}$.

For both Binomial and D-Regular GW models, the average level of synchronization $\langle Sync \rangle$ significantly decreased, as the distance  from the root  $h$ increased, for all considered choices of initial conditions $\theta_i(0)$ and natural frequencies $\omega_i$. The parameter regions of high synchronization (red-colored in the heatmaps, which correspond to $\langle sync \rangle >0.8 $) were only found for low $h$ values ($h\leq 4$ in the binomial model and $h\leq 2$ in the D-regular model). It is important to notice that the synchronization levels did not substantially increased even for higher values of the coupling strength $K$, which in our study was allowed to range from $0$ to $10$. This result indicates that desynchronization may be a predominant phenomenon in GW trees with the chosen average offspring values.

From the GW model perspective, a simple visual inspection indicates that the Binomial model promoted a higher synchronization for lower $h$ values than the D-Regular. For each one of the five combinations of the initial condition and natural frequencies, the comparison between models shows a significant higher synchronization in the binomial model. Further analysis would be required for us to understand the causes of such effect. However, we hypothesize that the variability in the offspring number that is present in the binomial model could be a major factor to explain this increase in synchronization.

\begin{figure}
	\centering
	\includegraphics[scale=0.5]{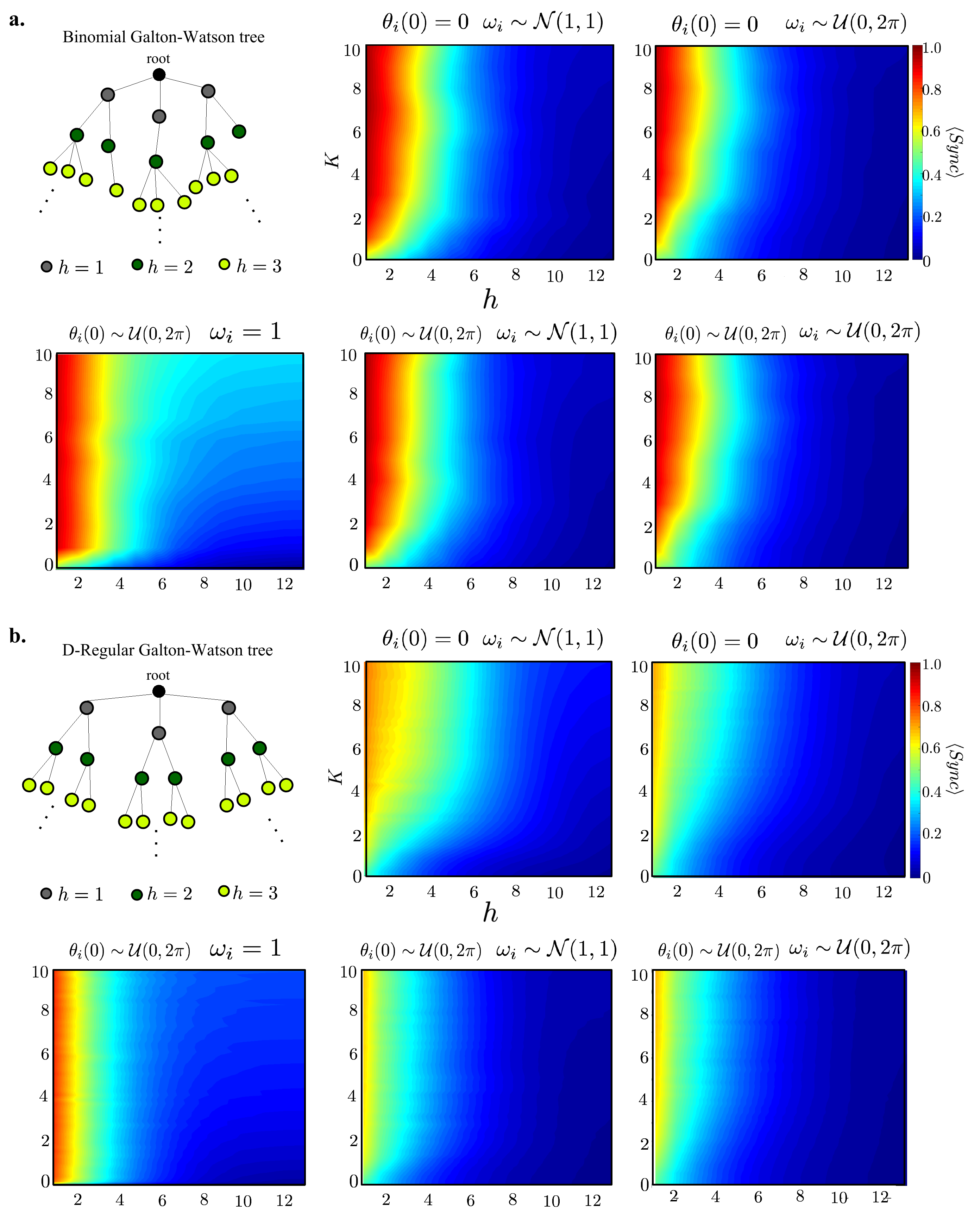}
	\label{Big}
\end{figure}

\newpage 
\noindent
\small{\textbf{Figure 1. Stochastic Kuramoto Dynamics in GW trees: average synchronization levels between the root and different nodes.}  \textbf{a.}  In the Binomial model, the offspring was given by a binomial random variable  with distribution $Bin(3,\frac{2}{3})$. \textbf{b.}  In the D-Regular model, the root has 3 children, while the other nodes have only 2 children.  In the top left of Figures \textbf{a.}  and  \textbf{b.}, we exhibit  schematic illustrations of both GW tree models, with nodes at different distances from the root. The parameter $K$ represents the coupling strength  between any two connected nodes (see text for details).  We assumed two distinct distributions for the initial phases $\theta_i(0)$ (all equal or uniformly distributed), while the natural frequencies $\omega_i$  were assumed equal to 1, or normally distributed with mean and variance equal to 1, or uniformly distributed on $[0,2\pi]$.  For different $K$ values, we estimated the loss of synchronization between the root vertex and those at different distances $h$, as $h$ increased from 1 to 13.  We observed clear transition patterns between high and low synchronization regions in the $h \times K$ plane.}






\appendix

\section{Weak local convergence and nice networks}\label{app:examples}

The goal of this section is to prove Proposition \ref{prop:mcontinuity}, Theorem \ref{theo:lwlexamples} and show that we have examples that satisfy the assumptions of Theorem \ref{theo:hydlimit} (cf. Remark \ref{rem:nicenet}).

\subsection{Proof of Proposition \ref{prop:mcontinuity}}
\label{proof:prop:mcontinuity}
Proposition \ref{prop:mcontinuity} is a simpler version of Theorem \ref{theo:exttheta}.

Start with a sequence of rooted graphs $\{(G_k,o_k)\}_{k\in\N}$ that converges in the local topology to $(G,o)\in\sG^*.$ We want to show that the respective laws of the random rooted networks, $(N^{G_k},o_k)$ also converges to the law of $(N^G,o)$, that is, $\sM(G_k,o_k)\to\sM(G,o)$ in the BL-distance.

\ignore{
The idea of the proof is the following. Two rooted graphs are close in the local topology if they are isomorphic for a large ball. Then  we can couple  the marks to be equal in this large ball to see that the networks are also close.
}
From the convergence $(G_k,o_k)\to (G,o)$ in $\sG^*$ we know that for any $r\in\N$ fixed $(G_k,o_k)_r=(G,o)_r$ for $k$ sufficiently large.

When we construct the random networks $N^G$ and $N^{G_k}$ we can couple the marks to be equal in $(G_k,o_k)_r=(G,o)_r$. In particular, $(N^{G_k},o_k)_r=(N^{G},o)_r$ almost surely.

Therefore, by definition and almost surely
\[\dist{\sN^*}{[N^G,o]}{[N^{G_k},o_k]}\leq \frac{1}{1+r}\]
for $k$ big enough. This bound immediately translates to a bound for their laws and we finish the proof.

\subsection{Proof of Theorem \ref{theo:lwlexamples}}
\label{proof:theo:lwlexamples}
Theorem \ref{theo:lwlexamples} is a simpler version of Theorem \ref{theo:hydlimit}.

We have a sequence of graphs $\{G_n\}_{n\in\N}$ that converges locally weakly to the measure $\rho\in\sG^*$. We want to show that defining the random networks $N_n:=N^{G_n}$ by adding the i.i.d. marks, we have that $U(N_n)\to\rho\sM$ almost surely in $\sP(\sN^*).$ We are assuming that $\max_{v\in [n]}d_v^{G_n}\leq n^{\eps_n}$ with $\eps_n\to 0$ as $n\to\infty.$

Let $\Ex{\,\cdot\,}$ be the expectation in the space where the marks of the sequence of the networks $(N_n)_{n\in\N}$ can be defined. We can easily check that \[\Ex{U(N_n)}=U(G_n)\sM.\] In particular, from the continuity of $\sM$ it is simple to deduce that $\Ex{U(N_n)}\to\rho\sM$.

So we just need to show that almost surely
\[\lim_{n\to\infty}\left|U(N_n)(h)-\Ex{U(N_n)(h)}\right|=0\]
for any borel measurable function $h:\sN^*\to\R$ with $\|h\|_{BL}\leq 1.$

The idea, as in Theorem \ref{theo:hydlimit}, is to concentrate $U(N_n)_r$ around its mean, where
\begin{align}\label{eq:restricted}
U(N_n)_r=\frac{1}{n}\sum_{v=1}^n\delta_{[N_n(v),v]_r}.
\end{align}

Throughout the remainder of this section we have a bounded measurable test function $h:\sN^*\to\R$ with $\blnorm{h}\leq 1.$

The same arguments in Section \ref{proof:prop:mcontinuity} says that
\[|U(N_n)(h)-U(N_n)_r(h)|\leq \frac{1}{1+r}\]
that also holds in mean:
\[|\Ex{U(N_n)(h)}-\Ex{U(N_n)_r(h)}|\leq \frac{1}{1+r}.\]

Now we will use Azuma's inequality (cf. Theorem \ref{thm:azuma}) to show that $U(N_n)_r(h)$ is concentrated around its mean.
From (\ref{eq:restricted})
\[U(N_n)_r(h)=\frac{1}{n}\sum_{v=1}^nh([N_n(v),v]_r)\]
that is a function of the independent variables $(\omega_v,\theta_v(0))_{v\in [n]}$ and $(\mu_e)_{e\in E_n},$ where $E_n$ is the edge set of $G_n$.

If we change one of these marks on the vertex $w$, the only networks that are changed are those $[N_n(v),v]_r$ with $w \in (G_n(v),v)_r$, that is, $v \in (G_n(w),w)_r$. Therefore, $U(N_n)_r(h)$ changes by at most the size of $(G_n(w),w)_r$ over $n$. By assumption $|(G_n(w),w)_r|\leq n^{(r+1)\eps_n}.$ Observe that $U(N_n)(h)$ is a function of $2n+E_n$ independent random variables. 

From Azuma inequality we have that
\[\Pr{|U(N_n)_r(h)-\Ex{U(N_n)_r(h)}|>t}\leq 2\expp{-\frac{2t^2n^2}{(2n+E_n)n^{2(r+1)\eps_n}}}.\]

To see that this bound is summable we use that the number of edges is one half of the sum of degrees to bound: \[E_n\leq \frac{1}{2}n\max_{v\in [n]}d_v^{G_n}\leq\frac{1}{2} nn^{\eps_n}.\]
Using the triangle inequality several times and the previous bounds we see that \[\lim_{n\to\infty}|U(N_n)(h)-\Ex{U(N_n)(h)}|=0\]
for any $h$ with $\blnorm{h}\leq 1.$
\subsection{Good examples}\label{sec:goodexamples}

In this section we want to see that we have examples that satisfy the assumptions in Theorem \ref{theo:hydlimit}. 

From Theorem \ref{theo:lwlexamples}, the Assumption \ref{item:lwc} in the Theorem \ref{theo:hydlimit} is satisfied almost surely for any of the graphs in Example \ref{exe:cycle}, \ref{exe:er} and \ref{exe:degree} when we add i.i.d. marks.

The Assumption \ref{item:degree} of Theorem \ref{theo:hydlimit} is satisfied trivially for graphs with bounded degree (Example \ref{exe:cycle}). We impose in Example \ref{exe:degree} that uniformly graphs with a given degree sequence also satisfy $\max_{v\in [n]}d_v^{G_n}\leq n^{\eps_n}$ with $\eps_n\to 0$ as $n\to\infty.$

Now we will check that $\ER$ graphs (Example \ref{exe:er}) also satisfy Assumption \ref{item:degree} of Theorem \ref{theo:hydlimit}, almost surely.

By the union bound, and Chernoff bound (Theorem 2.3.1, \cite{Vershynin})
\[\Pr{\max_{v\in [n]}d_v^{G_n}>n^{\eps_n}}\leq n\Pr{d_1^{G_n}>n^{\eps_n}}\leq \expp{-n^{\eps_n}\log n^{\eps_n}}\]
for $n$ big enough. Since this bound is summable (we can impose $n^{\eps_n}\to \infty$ as $n\to\infty$), a sequence of $\ER$ graphs satisfy the assumption almost surely.

It remains to see that in the Examples \ref{exe:cycle}, \ref{exe:er} and \ref{exe:degree}, the local weak limit is supported in nice graphs, that is, the local weak limit of theses sequences are supported in rooted graphs that satisfy
\begin{align}\label{eq:nicegraphs}
|\partial (G,o)_r|\leq ae^{ar}, \mbox{ for some }a>0.
\end{align}

It is clear that (\ref{eq:nicegraphs}) is achieved when the local weak limit measure is supported in graphs with uniform bounded degree. So we are done with Example \ref{exe:cycle}.

Now we will check (\ref{eq:nicegraphs}) for the GW tree. This will show that we are done with Examples \ref{exe:er} and \ref{exe:degree} because a Unimodular GW tree is equal to a GW tree after the first generation.

Given a probability distribution $P$ on $\mathbb{N}$ with mean $\mu \in (0,+\infty)$ consider $GW(P)$ the GW distribution on the set of rooted trees. It means that in each generation each individual has $i$ children with probability $P(i)$ independent of the other individuals. Let $Z_n$ be the number of individuals at generation $n$, $Z_0=1.$ We know that $Z_n/(\mu)^n$ is a martingale. In particular $\Ex{Z_n}=\mu^n.$ But we also have that $|\partial (G,o)_n|=Z_n$ if $(G,o)$ has distribution $GW(P).$

Therefore,
\begin{eqnarray}
\Pr{|\partial (G,o)_r|\geq (2\mu)^r} \leq \dfrac{\Ex{|\partial (G,o)_r|}}{(2\mu)^r}=\dfrac{1}{2^r}. \nonumber
\end{eqnarray}
which is summable in $r$. Therefore, $|\partial (G,o)_r|\leq e^{(\log 2\mu)r}$ for $r$ big enough, almost surely. Choosing $a$ big enough we see that $|\partial (G,o)_r|\leq ae^{ar}$ for any $r\geq 1$ and for almost all realizations of $(G,o).$

\section{Appendix -  Continuity for finite graphs}
\label{sec:limfinitenet}

The goal of this section is to justify (\ref{eq:lim-finitenet}):
\[d_{BL}(\ovl{\Theta}[N,o]_r,\ovl{\Theta}[N_n,o_n]_r)\to 0.\]
We think that this can be derived from the standard theory of Ordinary Differential Equations. For completeness we provide a prove in this section. Our notation follows that one in Section \ref{sub:prooftheo}.

We will use the following Lemma.

\begin{lemma}[Proof in Section \ref{sec:samegraph}]\label{le:samegraph} Consider a graph $G$ and two {\em finite} networks $N_i=(G,\vecg{\mu}^{(i)},\vecg{\omega}^{(i)},\vecg{\theta}^{(i)}(0))\in \sN$, $i=1,2.$ Suppose that all marks are close by $\eps:$ \[\sup_{vu \in E_G}|\mu_{vu}^{(2)}-\mu_{vu}^{(1)}|\leq \eps,\ \ \sup_{v\in V_G}|\omega_v^{(2)}-\omega_v^{(1)}|\leq \eps, \mbox{ and } \sup_{v\in V_G}|\theta_v(0)^{(2)}-\theta_v(0)^{(1)}|\leq \eps.\]
Then there exists a constant $\ovl{C}$ depending on $\blnorm{\phi}$, $\pnorm{Lip}{\psi},$ $T$, and $N$ such that almost surely with respect to the law of the Brownian motions
\[\sup_{v\in G}\sup_{t\leq T}\left(|\theta_v^{N_1}(t)-\theta_v^{N_2}(t)|\right)\leq \ovl{C}e(G)\eps.\]
\end{lemma}

From the convergence of $[N_n,o_n]$ to $[N,o]$ we know that for any $r$ and for $n$ large enough \[[N_n,o_n]_r=[N,o]_r.\]

So we can apply Lemma \ref{le:samegraph} with $N_1=(N,o)_r$, $N_2=(N_n,o_n)_r,$ and \[\eps=\dist{\sN^*}{[N,o]_r}{[N_n,o_n]_r}\] to obtain the bound
\begin{align}\label{eq:bounde}
\sup_{v\in (N,o)_r}\sup_{t\leq T}|\theta_v^{(N,o)_r}(t)-\theta_v^{(N_n,o_n)_r}(t)|\leq \ovl{C}e([N,o]_r)\dist{\sN^*}{[N,o]_r}{[N_n,o_n]_r}.
\end{align}

In that way, we use that the networks coincide at any radius and the bound in (\ref{eq:bounde}) to conclude that almost surely (cf. (\ref{def:distN}))
\[\dist{\sC^*}{[N,o]^{\theta}_r}{[N_n,o_n]^{\theta}_r}\leq \ovl{C}e([N,o]_r)\dist{\sN^*}{[N,o]_r}{[N_n,o_n]_r}\]
and the RHS of the last bound goes to zero as $n\to \infty.$ This is enough to finish.

\subsection{Proof of Lemma \ref{le:samegraph}}\label{sec:samegraph}
In this section, we prove Lemma \ref{le:samegraph}. The prove follows similar ideas of the proof of Lemma \ref{carne:thecorollary}.

For ease of notation we adopt the following conventions.
\begin{definition}\label{def:sameg:conv}
\begin{enumerate}
\item For the objects related to $N_2$ we omit the superscripts and just write $\vecg{\mu},$ $\vecg{\omega}$, $\vecg{\theta}(0)$ and $\vecg{\theta}.$
\item For the objects related to $N_1$ we omit the superscript and write a over-line writing  $\ovl{\vecg{\mu}},$ $\ovl{\vecg{\omega}}$, $\ovl{\vecg{\theta}}(0)$ and $\ovl{\vecg{\theta}}.$
\item We also write $P_{uv}=\mu_{uv}/\mu_v$ and $\ovl{P}_{uv}=\ovl{\mu}_{uv}/\ovl{\mu}_v$.
\item We write $I$ for the identity matrix.
\item $G$ has vertex set $V$ and edge set $E.$
\end{enumerate}   
\end{definition}

From Definition \ref{def:supersparseID} and for  $v\in V$, we have that
\begin{align} \label{eq:sdesame}
{\rm d}\theta_{v}(t)= \sum_{u \in V}P_{uv}(\phi(\theta_{u}(t),\theta_{v}(t);\omega_v,\omega_u){\rm d}t +\psi(\theta_{v}(t);\omega_v){\rm d}t+ {\rm d}B_v(t)
\end{align}
and the analogous formula holds for $\ovl{\theta}_v(\cdot)$ using the objects of $N_1$ but using the same Brownian motion.

 We suppose that the Brownian motions are coupled to be equal in each system.
\begin{proposition}\label{prop:deltasame} With our conventions, there exists a constant $C$ depending on $\blnorm{\phi}$, $\pnorm{Lip}{\psi}$, $T$, and $1/\mu_*$ such that 
\[|\theta_v(t)-\ovl{\theta}_v(t)|\leq C\left(\eps +\int_0^t\sum_{u\in V}(\ovl{P}+I)_{uv}|\theta_u(s)-\ovl{\theta}_u(s)|{\rm d}s\right)\]
for all $v\in V.$
\end{proposition}
\begin{proof}
From \ref{eq:sdesame},
\begin{eqnarray}
\theta_v(t)-\ovl{\theta}_v(t)=\theta_v(0)-\ovl{\theta}_v(0)+\int_{0}^{t} \Delta_1(v,s)+\Delta_2(v,s)+\Delta_3(v,s) {\rm d}s \nonumber
\end{eqnarray}
where
\begin{eqnarray}
\Delta_1(v,s)&:=&\sum_{u\in V}(P_{uv}-\ovl{P}_{uv})\phi(\theta_u(s),\theta_v(s);\omega_u,\omega_v),\nonumber \\
\Delta_2(v,s)&:=&\sum_{u\in V}\oP_{uv}(\phi(\theta_u(s),\theta_v(s);\omega_u,\omega_v)-\phi(\ovl{\theta}_{u}(s),\ovl{\theta}_{v}(s);\ovl{\omega}_u,\ovl{\omega}_v)),\nonumber \\
\Delta_3(v,s)&:=&\psi(\theta_v(s);\omega_v)-\psi(\ovl{\theta}_v(s);\ovl{\omega}_v).\nonumber
\end{eqnarray}
and the following bounds hold
\begin{eqnarray}
|\theta_v(0)-\ovl{\theta}_v(0)|&\leq &\eps \nonumber \\
|\Delta_1(v,s)|&\leq &\supnorm{\phi}\sum_{u\in V}|P_{uv}-\ovl{P}_{uv}|,\nonumber \\
|\Delta_2(v,s)|&\leq &\sum_{u=1}^n\oP_{uv}\pnorm{Lip}{\phi}(|\theta_v(s)-\ovl{\theta}_v(s)|+|\theta_u(s)-\ovl{\theta}_u(s)|+2\eps)\, \mbox{, and} \nonumber \\
|\Delta_3(v,s)|&\leq &\pnorm{Lip}{\psi}(|\theta_v(s)-\ovl{\theta}_v(s)|+\eps).\nonumber
\end{eqnarray}

Now we estimate $S_v:=\sum_{u\in V}|P_{uv}-\ovl{P}_{uv}|$. Remember that $\sup_{vu \in E}{|\mu_{vu}-\ovl{\mu}_{vu}|}\leq \eps.$ It is clear that $S_v=0$ if $v$ is an isolated vertex. If $v$ is not isolated, then
\begin{flalign*}
S_v&=\sum_{u\in V}\modulo{\frac{\mu_{uv}}{\mu_v}-\frac{\ovl{\mu}_{uv}}{\ovl{\mu}_v}}\\
&\leq \sum_{u\in V}\frac{1}{\mu_v}|\mu_{uv}-\ovl{\mu}_{uv}|+\ovl{\mu}_{uv}\modulo{\frac{\ovl{\mu}_v-\mu_v}{\ovl{\mu}_v\mu_v}}
\leq \frac{\eps d_v}{\mu_v} + \frac{|\ovl{\mu}_v-\mu_v|}{\mu_v}\leq \frac{2\eps d_v}{\mu_v}.
\end{flalign*}

The assumption that the network is nice says that $\mu_*\leq \mu_{uv}$ whenever $uv\in E$. That is, $d_v\mu_*\leq\mu_v.$

Combining these bounds we obtain the result.
\end{proof}

We now use Corollary \ref{gr:3rd} with

\begin{enumerate}
\item $\vec{u}(t):=(|\theta_v(t)-\ovl{\theta}_v(t)|)_{v \in V}.$
\item $\vec{a}(t):=\left(C\eps\right)_{v \in V},$ and each entry of this vector is non-negative.
\item $M(t)=M:=C(\oP+I),$ and this matrix does not depend on time $t$, it is entrywise non-negative, and it is finite dimensional since $G$ is finite.
\end{enumerate}

Therefore, for any $v \in V$
\begin{eqnarray} \label{eq:sameg:deltabound}
|\theta_v(t)-\ovl{\theta}_v(t)|&\leq & C\eps\expp{tC(\oP+I)}_v \nonumber
\end{eqnarray}

To relate this bound with the Random Walk in $(G,\ovl{\vecg{\mu}})$ observe that
\begin{eqnarray}
\expp{tC(\oP+I)}=e^{Ct}\expp{tC\oP}=e^{2Ct}e^{-Ct}\expp{tC\oP}. \nonumber
\end{eqnarray}

In the context of \cite[Section~5.1]{barlow2017random} we have that
\begin{flalign}\label{eq:heatkernel}
q_s(v,w):=\dfrac{1}{\ovl{\mu}_{w}}e^{-s}\expp{s\oP}_{vw}
\end{flalign}
 is the continuous time Heat Kernel for the Simple Random Walk in $(G,\ovl{\vecg{\mu}})$.

Therefore, since $q_s(v,w)\leq 1,$ for all $v,w\in V,$ then $e^{-s}\expp{s\ovl{P}}_{vw}\leq \ovl{\mu}_{w}\leq \mu^*d_w$ implies that
\begin{flalign*}
|\theta_v(t)-\ovl{\theta}_v(t)|\leq C\mu^*\eps e^{2CT}\sum_{w\in G}d_w=2C\mu^*e^{2CT}e(G)\eps.
\end{flalign*}
Notice that this bound is uniform in $t\in [0,T].$ This is enough to finish the proof.

\section{Linear Gronwall's inequality}
\label{subsec:gronwall}

\begin{proposition}[Corollary 2 of \cite{LinearGronwall}]\label{gr:2nd} Let the vector $\vec{a}(t)\in\R^n$ and the non-negative (entrywise) $n\times n$ matrices  $O(t)$, $M(t)$ be continuous functions of the single scalar variable $t$ for $t^{0}\leq t.$ Assume that $M(t)O(t)$ and $\int_{t^0}^{t}M(s)O(s)ds$ commute for $t^0\leq t$. If
\begin{eqnarray}
\vec{u}(t)\leq \vec{a}(t)+O(t)\int_{t^0}^{t}M(s)\vec{u}(s)ds, \ \ t^0\leq t \nonumber
\end{eqnarray}
then
\begin{eqnarray}
\vec{u}(t)\leq \vec{a}(t) + O(t)\int_{t^0}^{t}\expp{\int_{s}^{t}M(r)O(r)dr}M(s)\vec{a}(s)ds, \ \ t^0\leq t. \nonumber
\end{eqnarray}
\end{proposition}

\begin{corollary}\label{gr:3rd} Assume the hypothesis in Proposition \ref{gr:2nd} and additionally suppose that $O(t)=Id$ and $\vec{a}(t)$ is entrywise non-decreasing in each entry. In this case 
\begin{eqnarray}
\vec{u}(t)\leq \vec{a}(t)+\int_{t^0}^{t}M(s)\vec{u}(s)ds, \ \ t^0\leq t
\end{eqnarray}
implies
\begin{eqnarray}
\vec{u}(t)\leq \expp{\int_{t^0}^{t}M(s)ds}\vec{a}(t), \ \ t^0\leq t.
\end{eqnarray}
\end{corollary}
\begin{proof}
From Proposition \ref{gr:2nd}, and the fact that the exponential of a non-negative matrix is non-negative (every entry of each power is non-negative) we can compute
\begin{eqnarray}
\vec{u}(t)&\leq &\vec{a}(t) + \int_{t^0}^{t}\expp{\int_{s}^{t}M(r)dr}M(s)\vec{a}(s)ds \nonumber \\
(\vec{a}(s)\leq \vec{a}(t) \mbox{ whenever } s\leq t)&\leq & \vec{a}(t) + \left(\int_{t^0}^{t}\expp{\int_{s}^{t}M(r)dr}M(s)ds\right)\vec{a}(t) \nonumber \\
&= & \vec{a}(t) + \left(\int_{t^0}^{t}\dfrac{d}{ds}\expp{\int_{s}^{t}M(r)dr}ds\right)\vec{a}(t) \nonumber \\
&= & \vec{a}(t) -\vec{a}(t)+ \expp{\int_{t^0}^{t}M(s)ds}\vec{a}(t) \nonumber \\
&=& \expp{\int_{t^0}^{t}M(s)ds}\vec{a}(t). \nonumber 
\end{eqnarray}
\end{proof}

\bibliographystyle{plain} 

\bibliography{teseReis}
\end{document}